%% file: paper5.tex
\documentclass{conm-p-l}
\usepackage[utf8]{inputenc}
\usepackage{amsmath, amscd}
\usepackage{amssymb, bm} 
\usepackage{amsthm}
\usepackage{color}
\usepackage[dvipsnames]{xcolor}
\usepackage{cancel}
\usepackage{ulem}
\usepackage{geometry}
\usepackage[all]{xy}
\usepackage{mathrsfs}
\usepackage{mathtools}
\usepackage{tikz}
\usepackage{multicol}
\usepackage{soul}
\usepackage{graphicx, psfrag, tikz, tikz-cd}

\usepackage{caption}
\usepackage{subcaption}

\usepackage{paralist}

\numberwithin{equation}{section}

\newtheorem{theorem}{Theorem}[section]
\newtheorem{lemma}[theorem]{Lemma}

\theoremstyle{definition}
\newtheorem{definition}[theorem]{Definition}
\newtheorem{example}[theorem]{Example}

\theoremstyle{remark}
\newtheorem{remark}[theorem]{Remark}

\numberwithin{equation}{section}

\theoremstyle{corollary}
\newtheorem{corollary}[theorem]{Corollary}

\DeclareMathOperator{\Sp}{Sp}

\DeclareMathOperator{\Hom}{Hom}

\newcommand{\bC}{\mathbb C}
\newcommand{\bD}{\mathbb D}

\newcommand{\bR}{\mathbb R}

\newcommand{\bZ}{\mathbb Z}

\newcommand{\bj}{\mathsf{j}}

\newcommand{\cD}{\mathcal D}

\newcommand{\cM}{\mathcal M}

\newcommand{\critv}{\mathrm{critv}}
\newcommand{\Hor}{\mathrm{Hor}}
\newcommand{\Ver}{\mathrm{Vert}}
\def\re{\mathrm{Re}}
\def\im{\mathrm{Im}}

\newcommand{\la}{\lambda} 
\newcommand{\dd}{\partial}

\newcommand{\lra}{\longrightarrow} 

\definecolor{InkGreen}{RGB}{0,128,0}
\definecolor{Cblue}{RGB}{0, 150, 255}

\usepackage[colorlinks=true, linkcolor=blue, filecolor=blue,citecolor=blue]{hyperref}

\title[On fiber and base decompositions in the Fukaya category of a symplectic LG model]{On fiber and base decompositions in the Fukaya category of a symplectic Landau-Ginzburg model}

\author{Haniya Azam}
\address{Haniya Azam: Department of Mathematics, Lahore University of Management Sciences (LUMS), DHA, Lahore 54792, Pakistan}
\email{haniya.azam@lums.edu.pk}

\author{Catherine Cannizzo}
\address{Catherine Cannizzo: Columbia University, Department of Mathematics, New York, NY 10027, USA}
\email{cc5441@columbia.edu}
\thanks{C. Cannizzo was partially supported by NSF DMS-2316538}

\author{Heather Lee}
\address{Heather Lee: Pasadena, CA, USA}
\email{heatherlee.math@gmail.com}

\author{Chiu-Chu Melissa Liu}
\address{Chiu-Chu Melissa Liu: Columbia University, Department of Mathematics, New York, NY 10027, USA}
\email{ccliu@math.columbia.edu}

\subjclass[2020]{Primary 53D37}

\date{}

\begin{document}

\maketitle

\begin{abstract} In mirror symmetry, symplectic Landau-Ginzburg models are mirror to a large class of examples, in particular to Fano varieties and hypersurfaces of many Calabi-Yau and Fano varieties. When studying their Fukaya categories on the A-model in homological mirror symmetry, one needs to calculate the weights of pseudo-holomorphic discs bounded by Lagrangian branes. While these calculations simplify for exact and Lefschetz fibrations, we generalize the machinery for computing these weights by dropping the exact and Lefschetz assumptions. For a general symplectic Landau-Ginzburg model, a singular symplectic fibration, we prove that the weights and Lagrangian gradings split into base and fiber components. This is used in many calculations of Fukaya-Seidel categories to provide evidence of Kontsevich's homological mirror symmetry conjecture.
\end{abstract}

\tableofcontents

\vspace{12pt}

\section{Introduction}

The study of Landau-Ginzburg models has largely been motivated by mirror symmetry, thus far.  Mirror  symmetry was originally discovered for pairs of  {closed} Calabi-Yau manifolds, and it has since been generalized to other settings as well. For the known non-Calabi-Yau examples, the  mirror is best described not by just a manifold, but a Landau-Ginzburg (LG) model \cite{HV00}.  It is then necessary to understand the Fukaya category of a symplectic Landau-Ginzburg model for studying homological mirror symmetry. They appear in many works, for example \cite{seidel, AAhypersurf, Ca, CV}, and \cite[Appendix]{ushape}. {Defining and computing the Fukaya category are two separate steps. A definition of a general Fukaya category with B-field for transversely intersecting Lagrangians is given in \cite{ACLLb}. A definition of the Fukaya category of a symplectic Landau-Ginzburg model with wrapping is given in \cite[\textsection 3]{AAhypersurf}. In this paper, we provide tools for computing the Fukaya category. LG models are particularly amenable to calculations in the Fukaya category.} We generalize calculations in Fukaya categories of LG models by dropping the exact and Lefschetz assumptions. We prove two theorems, one for disc weights and one for Lagrangian gradings,  that they split into base and fiber contributions. Before stating our main results, we introduce symplectic LG models and their Fukaya categories. 

\begin{definition}[Symplectic Landau-Ginzburg model $(Y,v)$]
Suppose $Y$ is a symplectic manifold with a symplectic form $\omega$ and a smooth map
    $$
    v:Y \to \bC
    $$ 
    such that the set of critical values $\critv(v)$ of $v$ is finite.  Denote by $C :=\bC\backslash \{\critv(v)\}$ the set of regular values of $v$. 
 We further assume that $Y_C:= v^{-1}(C) \to C$ is a symplectic fibration. {That is, it is a locally trivial fibration with symplectic generic fiber and structure group contained in the symplectomorphism group of the fiber.} In particular, for each $c\in C$, the fiber $Y_c=v^{-1}(c)$ is a symplectic submanifold with symplectic form $\omega_c=\omega|_{v^{-1}(c)}$.    Then $(Y,v)$ is called a \emph{symplectic Landau-Ginzburg model.}  
\end{definition}

\begin{remark}\label{rmk:convexity} One generally also imposes the condition that $\omega$ be convex at infinity in order to define Floer theory. This is to ensure that pseudo-holomorphic discs do not escape to infinity. The results in this paper may be applied under the assumption that a maximum principle holds in the {fibers of the} LG model of interest {because a maximum principle holds in the base as mentioned below Corollary \ref{cor:section_or_fiber}}. Our results do not pertain to holomorphic discs near infinity. 
\end{remark}

\begin{remark} The fibration $v$ here is not assumed to be exact, i.e. $[\omega]=0$.  We also do not assume it to be  Lefschetz, in which case the critical locus is isolated and nondegenerate, meaning that the neighborhood of each singular fiber is locally modeled on the neighborhood of the zero fiber of $\sum_{j=1}^n z_j^2: \bC^n \to \bC$. (When $n=2$, $z_1^2+z_2^2$ can be written as $xy$ via a change of variable.)  In the local model for a Lefschetz fibration, the smooth fibers are isomorphic to $T^*S^{n-1}$.  The zero section is a Lagrangian sphere in the fiber known as the vanishing cycle.  The singular fiber is obtained from the smooth fibers by collapsing the vanishing cycle to the critical point.  
\end{remark}

 Given a Lagrangian submanifold in a smooth fiber, we can obtain a Lagrangian submanifold in the total space via the symplectic parallel transport.   Following \cite[\S 6]{symp_intro}, given a symplectic fibration $Y_C\to C$, for each $p\in Y_C$, denote by 
\begin{equation}\label{eq:vertical def}
\Ver_p:=\ker dv_p=T_pY_{v(p)}
\end{equation} the tangent space to the fiber. It determines a symplectic \emph{horizontal distribution}, $\Hor=\left(\ker dv\right)^{\perp_\omega}$, given by the symplectic orthogonal to the tangent spaces of fibers, i.e.
\begin{equation}\label{eq:hor def}
    \Hor_p:=(T_pY_{v(p)})^{\perp_\omega}\subset T_pY.
\end{equation}
The surjective linear map $dv_p: T_p Y \to T_{v(p)}C $ restricts to a linear isomorphism $dv_p|_{\Hor_p}: \Hor_p \to T_{v(p)}C$. Given any $\xi\in T_{v(p)}C$, we define the \emph{horizontal lift} to be $\xi^\sharp_p := (dv_p|_{\Hor_p})^{-1}(\xi) \in \Hor_p$. A path $\beta: \bR \to Y_C$ in $Y_C$ is said to be \emph{horizontal} if $\dot{\beta}(t)\in \Hor_{\beta(t)}$ for all $t$. 
Given  any path $\gamma: \bR \to C$ in the base $C$ and any  $t_0, t_1 \in \bR$, 
the \emph{symplectic parallel transport} map \begin{equation}\label{eq:par_transp}
\Phi_{\gamma(t_0)\xrightarrow[]{\gamma} \gamma(t_1)}: Y_{\gamma(t_0)}\to Y_{\gamma(t_1)}
\end{equation}
is defined as follows. Given any $p\in Y_{\gamma(t_0)}$, there is a unique horizontal curve $\gamma^\sharp:\bR\to Y_C$ such that $v\circ \gamma^\sharp = \gamma: \bR\to C$ and  $\gamma^\sharp(t_0)=p$. 
Define $\Phi_{\gamma(t_0)\to \gamma(t_1)}(p):= \gamma^\sharp(t_1) \in Y_{\gamma(t_1)}$. 
It is known that $\Phi_{\gamma(t_0)\to \gamma(t_1)}$ is a symplectomorphism from $(Y_{\gamma(t_0)}, \omega_{\gamma(t_0)})$
to $(Y_{\gamma(t_1)}, \omega_{\gamma(t_1)})$. This allows us to construct fibered Lagrangians in $Y$ defined below. {When the path $\gamma$ is already indicated in the starting and ending points, we may drop it and write $\Phi_{\gamma(t_0)\to \gamma(t_1)}$.}

{\bf Objects.} In the aforementioned references \cite{seidel, AAhypersurf, Ca, CV}, the objects considered in the Fukaya category are Lagrangians fibered over and parallel along curves which do not bound discs, \cite[Equation (3.37)]{AAhypersurf}.

\begin{definition}[Fibered Lagrangians] \label{def: fibered Lagrangian}
Let $\gamma: \bR\to C$ be a smooth embedded curve in $C $.
A \emph{Lagrangian in $Y$ fibered over $\gamma$} is a Lagrangian submanifold $L$ of $(Y,\omega)$ such that 
\begin{enumerate}
    \item $v(L)= \gamma(\bR)$, 
    \item  $v|_L: L\to \gamma(\bR)$ is a locally trivial fibration, and
    \item  $L_c := L\cap Y_c = (v|_L)^{-1}(c)$ is a Lagrangian submanifold
    of $(Y_c, \omega_c)$ for any $c\in \gamma(\bR)$.
    \end{enumerate} 
A Lagrangian $L\subset Y_C\subset Y$ 
fibered over $\gamma:\bR\to C\subset \bC$  is \emph{parallel along $\gamma$} if for any $t \in \bR$, 
\[
L_{\gamma(t)} =\Phi_{\gamma(0)\xrightarrow[]{\gamma}\gamma(t)}(L_{\gamma(0)}), 
\]
or equivalently, for any $t_1, t_2\in \bR$, 
\[
L_{\gamma(t_2)}= \Phi_{\gamma(t_1)\xrightarrow[]{\gamma} \gamma(t_2)}(L_{\gamma(t_1)}).
\]
We say a Lagrangian is \underline{fibered} (and \underline{parallel}) if it is fibered over (and parallel along)  $\gamma$. When working with multiple fibered Lagrangians, we denote
$$
L_{j,c}:=(L_j)_c.
$$
\end{definition}

Morphisms in the Fukaya category are defined by Floer theory for transversely intersecting Lagrangians \cite[\textsection 3]{ACLLb}. To define them for LG models, we need some control on the base curves. In Remark \ref{rmk:admissible} and the paragraph that follows, we motivate the choice of $\gamma$. Then we state Lemma \ref{lem:isotopy_intro} relating fibered Lagrangians for different choices of $\gamma$. In Remark \ref{rmk:admissible flow}, we describe the procedure to move fibered Lagrangians so they intersect transversely in the base.

\begin{remark}\label{rmk:admissible} In order to define Floer theory for symplectic LG models, the Lagrangian objects need to satisfy some additional admissibility conditions with respect to $v$, outside of a compact set, so we can have a good control over the holomorphic discs.  The notion of admissibility was originally introduced by Kontsevich. There are many models of admissibility, and for the examples in this paper, we will adapt the notion that an \emph{admissible Lagrangian} $L\subset Y$ is one whose image $v(L)$ outside of a compact set in $\bC$ is a union of radial rays in the complement of the negative real axis (this model was used in for example \cite[{Definition 3.5}]{AAhypersurf}). The results in this paper work for any models of admissibility, so long as the Lagrangians are fibered, since as mentioned in Remark \ref{rmk:convexity}, our results do not pertain to holomorphic discs near infinity.
\end{remark}
{If $Y$ is Liouville}, a {Lagrangian fibered over and parallel along
a curve} is {an example of} a thimble-like Lagrangian in \cite[Definition 8.19]{GPS22} as a {variant of the notion of thimbles} in Lefschetz fibrations. When $v$ is Lefschetz, a thimble $L$ is a {Lagrangian} obtained by parallel transport of a vanishing cycle in a smooth fiber, over a curve $\gamma$ in the base with one end at a critical value.  In this case, the vanishing cycle converges to the critical point under parallel transport, and $L$ is a smooth Lagrangian, in particular it is still smooth at the critical point. If the singular fiber has a more complicated kind of singularity (e.g. locally modeled by ${z_1 z_2 z_3}: \bC^3\to \bC$), then $\gamma$ should avoid the critical value for $L$ to be smooth; see Example~\ref{ex:Harvey-Lawson} below.  In this case, one can instead construct Lagrangians that are fibered over U-shaped curves in the base \cite{AAhypersurf}.  Outside of a compact set in $\bC$ containing $\critv(v)$, each U-shaped curve consists of two radial rays; see Figure \ref{fig: ushape_paper5}.

\begin{example}\label{ex:Harvey-Lawson}
Let $Y=\bC^3$ be equipped with the standard symplectic and complex structures, and define $v:\bC^3\to \bC$ by $v(z_1,z_2,z_3)=z_1z_2z_3$. Then $(\bC^3, v)$ is a symplectic Landau-Ginzburg model, $0$ is the only critical value of $v$, and $\mathrm{crit}(v)\subset \bC^3$ is the union of the three coordinate axes. For any $a\in [0,\infty)$, define
$$
L_a:= \{(\sqrt{a+|u|^2}e^{i \theta}, u, \bar{u}e^{-i \theta}): u\in \bC, e^{i \theta}\in S^1\} \subset \bC^3.
$$
Observe that $v(\sqrt{a+|u|^2}e^{i\theta}, u, \bar{u} e^{-i\theta})= \sqrt{a+|u|^2}|u|^2$ which is a non-negative real number independent of $\theta$, and that $v(L_a)=[0,\infty)$. For any $a\in [0,\infty)$ and any $c\in (0,\infty)$, $L_{a,c} = L_a \cap v^{-1}(c)$ is a Lagrangian submanifold of $Y_c=v^{-1}(c)\cong (\bC^*)^2$, and 
$L_{a,c}$ is diffeomorphic to $T^2=S^1\times S^1$. 
\begin{itemize}
    \item When $a>0$, $L_a$ is a smooth Lagrangian submanifold of $\bC^3$, $L_{a,0}=L_a\cap v^{-1}(0)$ is a circle of radius $\sqrt{a}$ in the $z_1$-axis, and $L_a$ is diffeomorphic to $\bR^2\times S^1$
    which is homeomorphic to $C(S^1)\times S^1$, where $C(S^1)$ denotes a cone over $S^1$.
    \item When $a=0$, $L_{0,0}=L_0\cap v^{-1}(0)$ consists of a point $(0,0,0)$, which is the only singular point in $L_0$, and $L_0$ is homeomorphic to $C(S^1\times S^1)=C(T^2)$.  
\end{itemize}
\end{example}

The Lagrangians $L_a\subset \bC^3$ in Example~\ref{ex:Harvey-Lawson} are special cases of 
Harvey-Lawson special Lagrangians in $\bC^n$  \cite[III.3.A]{Harvey-Lawson}
and Aganagic-Vafa Lagrangians
in symplectic toric Calabi-Yau threefolds \cite{Aganagic-Vafa}.

\begin{definition}[Fibered Lagrangian isotopies]\label{def: fibered Lagrangian isotopies}
Let 
$$
h:\bR \times [0,1]\to C\subset \bC, \quad (t,s)\mapsto h(t,s)=:\gamma_s(t)
$$
be an isotopy of embedded curves in $C$, which is a Lagrangian 
isotopy in $C\subset \bC$. A {\em Lagrangian isotopy fibered over $h$} is a
Lagrangian isotopy 
$$
\psi: L\times [0,1]\to Y, \quad (p,s)\mapsto \psi(p,s) =: \psi_s(p)
$$
such that for any $s\in [0,1]$,  $\psi_s(L)$ is a Lagrangian fibered over the embedded curve 
$\gamma_s:\bR\to C\subset \bC$. We call such an isotopy a \emph{fibered Lagrangian isotopy}.
\end{definition}

Let $\gamma:\bR\to C\subset \bC$ be an embedded curve
in $C$.  Given any $t_0\in \bR$ and any Lagrangian submanifold 
$\ell \subset Y_{\gamma(t_0)}$ there exists a unique Lagrangian $L \subset Y_C \subset Y$
fibered over and parallel along $\gamma$ such that
$L_{\gamma(t_0)} = \ell$. Let
\[
h: \bR\times [0,1]\to C \subset \bC, \qquad
(t,s)\mapsto h(t,s)  =: \gamma_s(t)
\]
be a Lagrangian isotopy in $C\subset \bC$ such that $\gamma_0=\gamma$
and $h(t_0,s)= h(t_0,0) =\gamma(t_0)$ for any $s\in [0,1]$. Then  
there is a unique Lagrangian isotopy 
\[
\psi: L\times [0,1]\to Y_C\subset Y, \qquad (p,s)\mapsto \psi(p,s)= \psi_s(p)
\]
fibered over $h$ satisfying the following conditions. 
\begin{enumerate} 
\item $\psi_0(p)=p$ for any $p\in L$. 
\item $\psi_s(p)=p$ for any $p\in \ell\subset Y_{\gamma(t_0)}$ and any $s\in [0,1]$. 
\item For any $s\in [0,1]$, $\psi_s(L)$ is a Lagrangian fibered over and 
parallel along the embedded curve $\gamma_s:\bR\to C$. 
\end{enumerate}
We observe that, (2) and (3) imply (1). 

More explicitly, 
\begin{equation}\label{eq:parallel isotopy}
\psi(p,s) = 
\Phi_{\gamma(t_0)=\gamma_s(t_0) \xrightarrow[]{\gamma_s}\gamma_s(t)}\Phi_{\gamma(t)\xrightarrow[]{\gamma} \gamma(t_0)}(p).
\end{equation}
\begin{remark}
    We do not assume the symplectic connection is flat. Therefore, the choice of path between two points we parallel transport over matters. In order to move between a collection of parallel fibered Lagrangians in an isotopy in a LG model, we therefore must move back to the starting point at $\gamma(t_0)$ in order to move to a different Lagrangian in the collection.
\end{remark}

One result we prove in \textsection \ref{sec:area_triv} is that even though $[\omega]\neq 0$, that is $\omega$ is not exact, the Lagrangian isotopies one considers in an LG model are exact. See Lemma \ref{lem:Lag_isot_area} for the proof. The reason, informally, is because the Lagrangian isotopy fixes a  Lagrangian in a reference fiber over $\gamma(t_0)$, and because fibered Lagrangians are fibered over contractible paths in the base. Therefore, $L \to v(L)$ is a trivializable fibration. For parallel fibered Lagrangians, parallel transport $\ell \times \mathbb{R} \ni (p,t) \mapsto \Phi_{\gamma(t_0) \xrightarrow[]{\gamma} \gamma(t)}(p) \in L_{\gamma(t)}\subset L$ provides a trivialization.
\begin{lemma} \label{lem:isotopy_intro}
Let $L$ be a Lagrangian in $Y$ fibered over an embedded curve $\gamma: \bR\to C\subset \bC$. 
Let 
\[
h: \bR\times [0,1] \to C \subset \bC, \qquad
(t,s)\mapsto h(t,s) = \gamma_s(t)
\]
be an isotopy in $C$ such that $\gamma_0=\gamma$.
Let 
\[
\psi: L\times [0,1]\to Y_C\subset Y, \qquad
(p,s)\mapsto \psi(p,s)=\psi_s(p)
\]
be a Lagrangian isotopy in $Y_C$ fibered over $h$ such that $\psi_0(p)=p$ for all $p\in L$.

If there exists $t_0\in \bR$
such that for any $s\in [0,1]$,
$h(t_0,s)= \gamma(t_0)$ and $\psi_s(p)=p$ for any $p\in L_{\gamma(t_0)}$, then $\psi$ is an exact Lagrangian isotopy.
\end{lemma}

\begin{remark}The reason for all Lagrangians to pass through the same point in the base is to have a reference fiber to apply a cup functor, in applications. One reason to  choose that point to be $\gamma(t_0)= -\epsilon$ is that, when $Y$ is equipped with a holomorphic volume form $\Omega$ and $v:Y\to \bC$ is holomorphic, a necessary condition for the convergence of the oscillatory integral $\int_L e^{-v}\Omega$ along a Lagrangian $L\subset Y$
is that $\Re v(z) \to +\infty$ as $|z|\to +\infty$ for $z\in L$.  
\end{remark}

{\bf Morphisms.} To define a morphism between any two Lagrangian objects $L_0$ and $L_1$ in the Fukaya category of a symplectic Landau-Ginzburg model, we need to first position them properly using Hamiltonian perturbations as briefly explained in Remark \ref{rmk:admissible flow}.  For the notation of this paper, we assume that $L_0$ and $L_1$ are already properly positioned, then we can define morphisms  by the usual Floer cochain complexes $CF(L_0, L_1)$ generated by intersection points of transversely intersecting Lagrangians \cite[\textsection 3]{ACLLb}. Similarly, assuming that $L_0, \ldots, L_k$ are already in correct position, we have the $A_\infty$-product  $\mu^k:CF(L_{k-1}, L_k)\otimes \cdots \otimes CF(L_0,  L_1)\to CF(L_0, L_k)$, \cite[Definition 5.13]{ACLLb}.

{Morphisms in a Fukaya category of a symplectic LG model are defined by localization at continuation maps, \cite[Equation (3.43)]{AAhypersurf}. Continuation maps count discs bounded by Lagrangian isotopies. Lagrangian isotopies in a symplectic LG model fiber over isotopies of curves in the base $\mathbb{C}$. By \cite[Corollary 3.24]{AAhypersurf}, morphisms are computed by a colimit over linear wrappings clockwise on the projection of the output Lagrangian (or equivalently by linear wrappings counterclockwise on the projection of the input Lagrangian). By \cite[Proposition 6.1.9]{OT24} applied to the base $\mathbb{C}$, one can equivalently perturb by a single quadratic wrapping counterclockwise on the projection of the input Lagrangian. We describe this next.}

\begin{remark}\label{rmk:admissible flow} Given any two admissible fibered Lagrangians $L_0$ and $L_1$, the Floer cochain complex $CF(L_0, L_1)$ might not be well-defined if they do not intersect transversely, or even when they do, $HF(L_0, L_1)$ might not be invariant under Hamiltonian isotopy because of the noncompact ends. Therefore, the morphism between them in the Fukaya category of $(Y,v)$ is defined to be  $CF(\varphi_{\theta_0}(L_0), L_1)$, where $\varphi_{\theta_0}$ is the flow of a Hamiltonian isotopy that preserves the admissibility of the Lagrangians.  If we use the model where  admissible Lagrangians project to radial rays outside of a compact set as introduced in Remark \ref{rmk:admissible}, then we can choose $\varphi_{\theta_0}$ to be a Hamiltonian perturbation that rotates  the ends of $v(L_0)$ to angles greater than that of $v(L_1)$, while still avoid passing the negative real axis as in Figure \ref{fig: ushape_paper5}.  Similarly, the product maps $\mu^k:CF(\varphi_{\theta_{k-1}}(L_{k-1}),L_k)\otimes\cdots \otimes CF(\varphi_{\theta_0}(L_0), \varphi_{\theta_1}(L_1))$ are defined using the Hamiltonian perturbed Lagrangians so the ends of $v(L_j)$ get rotated counterclockwise to angles greater than that of $v(L_{j+1})$; see Figure \ref{fig: ushape_paper5} and the leftmost picture in  Figure \ref{fig: prism_complex_paper5} for  illustrations of $\mu^1$ and $\mu^2$, respectively. 
\end{remark}

\begin{figure}[h]
\centering
\fontsize{15pt}{17pt}\selectfont
\scalebox{0.65}{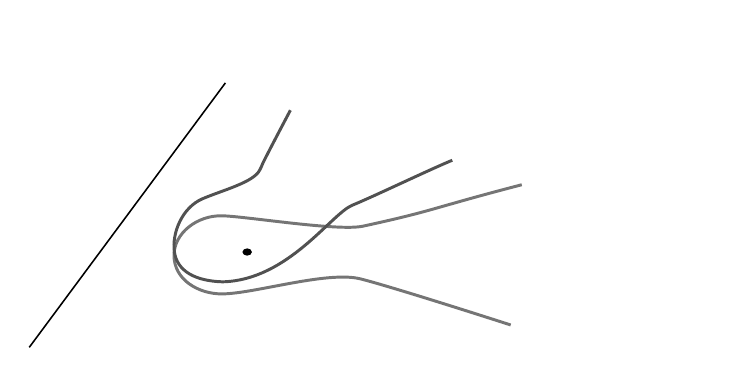}
\caption{An example of two U-shaped curves in the base of $v:Y\to \bC$ (with the black dot being a critical value), over which  fibered Lagrangians $L_0$ and $L_1$ are defined.  The radial ends of the base curve $v(L_0)$ are positioned counterclockwise away from those of $v(L_1)$. The intersections between $L_0$ and $L_1$ are in the fibers $Y_{c_+}$ and $Y_{c_-}$ above the points $c_+, c_- \in v(L_0)\cap v(L_1)$.  The shaded gray area is the image under $v$ of a $J$-holomorphic bigon contributing to $\mu^1$.}
\label{fig: ushape_paper5}
\end{figure}

{\bf Structure maps.} To compute $\mu^k$, one useful strategy is to deform a Lagrangian object  to one for which $\mu^k$ is easier to compute. In \cite{ACLLb}, we studied how $\mu^k$ changes as we perform this deformation by using an arbitrary Lagrangian isotopy. More specifically, we studied the change in the symplectic area, for example,  of a holomorphic disc contributing to the count of $\mu^k$ with boundaries on the Lagrangians involved. {As a consequence of Lemma \ref{lem:isotopy_intro}}, we have applications such as showing that the area of a holomorphic triangle (away from the singular fibers) contributing to $\mu^2$ splits into vertical and horizontal contributions. See Theorem \ref{cor:triangle disc area} for the precise statement. 

\begin{theorem}[Disc areas in $\mu^2$ decompose into vertical and horizontal contributions]
   Let $u$ be a pseudo-holomorphic triangle counted in $\mu^2:CF(L_{1}, L_2)\otimes CF(L_0,  L_1)\to CF(L_0, L_2)$ for fibered Lagrangians $L_0,L_1,L_2$ in a symplectic Landau-Ginzburg model. Then its area splits into a sum of fiber and base contributions.
\end{theorem}

{\bf Gradings on morphisms.} In Section \ref{sec:Maslov_fibrn}, under the additional assumptions that $Y$ is K\"{a}hler, $v: Y\to \bC$ is  holomorphic, and that there is a nowhere vanishing holomorphic volume form on $Y$, we discuss $\bZ$-gradings on the Floer cochain complexes of the form $CF(L_0, L_1)$ given two fibered Lagrangians $L_0$ and $L_1$. Following \cite{seidel_lagr}, we equip each $L_j$ with a grading structure, which is a choice of a lift $\widetilde \alpha_{L_j}:L_j\to \bR$ of the squared phase function $\alpha_{L_j}: L_j\to U(1)$.  Such a grading exists if and only if $\alpha_{L_j}$ is homotopic to a constant map. The graded Lagrangians $\widetilde L_j=(L_j, \widetilde \alpha_{L_j})$, $j=0,1$, then determine a canonical $\bZ$-grading on $CF(L_0, L_1)$.  

In our case of fibered Lagrangians in a symplectic Landau-Ginzburg model, we show that each $\widetilde \alpha_L$ can be written as a sum of a choice of grading $\widetilde \alpha_{v(L)}$ on the base curve $v(L)\subset \bC$ and a choice of relative grading $\widetilde \alpha_L^{\Ver}: L\to \bR$  (which is defined using the vertical subbundle of the tangent bundle, and it is uniquely determined given a choice of grading on a Lagrangian $L_c=L\cap Y_c\subset Y_c$ in a fiber).  In turn, we show that they determine a $\bZ$-grading on the Floer complexes given by a sum of base and fiber contributions. See Theorem \ref{cor_Maslov_paper5} for the precise statement.

\begin{theorem}[{Gradings for Floer complexes decompose into vertical and horizontal contributions}]
    Consider two fibered Lagrangians $L_0, L_1 \subset Y_C$ and $p\in L_{0,c} \pitchfork L_{1, c}$, so $c=v(p)$. Assume the fiber Lagrangians $L_{j,c}$ admit gradings for $j=0,1$. Then $L_0, L_1$ admit gradings $\widetilde{L}_0$, $\widetilde{L}_1$ such that
    $$
    \deg(\widetilde{L_0}, \widetilde{L_1};p)=\deg_{Y_c}(\widetilde{L}_{0,c}, \widetilde{L}_{1,c}; p)+\deg_{\bC}(\widetilde{v(L_0)},\widetilde{v(L_1)}; c).
    $$
\end{theorem}

Following these results regarding gradings on Lagrangians and their intersection points, we state and prove a result on the index of a pseudo-holomorphic bigon in an LG model, in Lemma \ref{bigon_Maslov_paper5}. Note that in literature on Fukaya categories the \emph{Maslov index} is defined in three contexts: for Lagrangians, for their intersection points, and for pseudo-holomorphic discs.

\subsection*{Acknowledgement} We thank Denis Auroux and Peter Samuelson for helpful conversations on gradings. C.-C.-M. Liu wishes to thank Marco Aldi, Allison Moore, and Nicola Tarasca, the organizers of the Richmond Geometry Meeting 2024, for the invitation to speak at the conference, and for the invitation to submit an article to its proceedings. We thank the anonymous referees for their very helpful comments.

\section{Areas of discs away from the singular fiber}\label{sec:area_triv}

In this section we provide a tool for computing $\mu^2$, the composition map in the Fukaya category. Let $(Y,\omega,v)$ be a symplectic Landau-Ginzburg model. All calculations here are assumed to happen over a {relatively compact} open subset of the base of $v|_{Y_C}:Y_C \to C$ trivializable by a diffeomorphism. {See the first image in Figure \ref{fig: prism_complex_paper5} for a depiction of a curve $u$ counted in $\mu^2$ in an example where Lagrangians are fibered over U-shaped curves. The subsequent images show a Lagrangian isotopy in the base. We may consider a subset of $v(\psi_1(L_0))$ depicted, so that the path in the base is still an embedding.}

      A \emph{Lagrangian isotopy} is a homotopy $\psi:L \times [0,1] \to Y$ such that $\psi_s(L):=\psi(L,s)$ is Lagrangian and a smooth embedding for each $s \in [0,1]$. When $\psi$ is a Lagrangian isotopy, we have  $\psi^*\omega = b \wedge ds$, where $b$ is {a real 1-form} on $L\times [0,1]$ that keeps track of the change in the symplectic form under the isotopy. More explicitly, in  local coordinates $(x_1,\ldots, x_n)$ on $L$
      and the global coordinate $s$ on $[0,1]$, we have
\begin{equation}
\psi^*\omega =\sum_{i<j} a_{ij}(x,s) dx_i \wedge dx_j+ \sum_i b_i(x,s) dx_i \wedge ds
\end{equation}
and 
\begin{equation}
\psi_s^*\omega = \sum_{i<j} a_{ij}(x,s) dx_i\wedge dx_j =0
\end{equation}
since $\psi_s(L)$ is a Lagrangian. Note that
$b\wedge ds = (b + f ds)\wedge ds$ for any smooth function 
$f(x,s)$ on $L\times [0,1]$. We may choose 
\begin{equation} \label{eqn:b_Lag_isot}
b =\sum_i b_i(x,s) dx_i = 
-\iota_{\frac{\partial}{\partial s}}\psi^*\omega.
\end{equation}
For each fixed $s$, $b_s$ is a closed 1-form on $L$.  A Lagrangian isotopy is \emph{exact} if  $b\wedge ds=d(Hds)$ for some smooth function $H: L\times [0,1]\to \bR$, in which case each $b_s=dH_s$ is an exact 1-form on $L$, and
$$
dH = -\iota_{\frac{\partial}{\partial s}}\psi^*\omega  + \frac{\partial H}{\partial s}ds. 
$$
See \cite{Seidel_isotopy}  and \cite[Section 6]{ACLLb}  for more details.  

{We do not assume the LG model is exact. However, a Lagrangian isotopy of a fibered Lagrangian $L$ is exact if it fixes both a point $\gamma(t_0)$ in the base curve $v(L)$ and the  Lagrangian in the fiber over $\gamma(t_0)$. We do not assume the fibered Lagrangians are parallel for this result.}

\begin{lemma} \label{lem:Lag_isot_area}
Let $L$ be a Lagrangian in $Y$ fibered over an embedded curve $\gamma: \bR\to C\subset \bC$. 
Let 
\[
h: \bR\times [0,1] \to C \subset \bC, \qquad
(t,s)\mapsto h(t,s) = \gamma_s(t)
\]
be an isotopy in $C$ such that $\gamma_0=\gamma$.
Let 
\[
\psi: L\times [0,1]\to Y_C\subset Y, \qquad
(p,s)\mapsto \psi(p,s)=\psi_s(p)
\]
be a Lagrangian isotopy in $Y_C$ fibered over $h$ such that $\psi_0(p)=p$ for all $p\in L$.

If there exists $t_0\in \bR$
such that for any $s\in [0,1]$,
$h(t_0,s)= \gamma(t_0)$ and $\psi_s(p)=p$ for any $p\in L_{\gamma(t_0)}$, then $\psi$ is an exact Lagrangian isotopy.
\end{lemma}
\begin{proof} 
We have $\psi^*\omega = b\wedge ds$, where
$b = -\iota_{\frac{\partial}{\partial s}}\psi^*\omega$
by Equation \eqref{eqn:b_Lag_isot}. 
Let  
$$
\iota_L: L_{\gamma(t_0)} \to L,\quad
\iota_Y: L_{\gamma(t_0)} \to Y_{\gamma(t_0)} \to Y
$$
be inclusion maps, and let $\pi: L_{\gamma(t_0)}\times [0,1]\to L_{\gamma(t_0)}$ be the projection to the first factor. 
By assumption, for any $(p,s)\in L_{\gamma(t_0)}\times [0,1]$, 
$$
\psi \circ (\iota_L \times id_{[0,1]})(p,s)
=\psi(p,s)= p =\iota_Y(p) = \iota_Y\circ \pi(p,s), 
$$
so $\psi\circ(\iota_L \times id_{[0,1]}) = \iota_Y\circ \pi: L_{\gamma(t_0)}\times [0,1] \to Y$. Thus
$$
(\iota_L\times id_{[0,1]})^*\psi^*\omega
= \left(\psi\circ (\iota_L\times id_{[0,1]})\right)^*\omega
=(\iota_Y \circ \pi)^*\omega
=\pi^* \iota_Y^*\omega =0
$$
since $\iota_Y^*\omega=0$. Therefore, 
$$
(\iota_L\times id_{[0,1]})^*b = 
- (\iota_L\times id_{[0,1]})^*(\iota_{\frac{\partial}{\partial s }}\psi^*\omega)
= -\iota_{\frac{\partial}{\partial s}}
(\iota_L\times id_{[0,1]})^*\psi^*\omega =0
$$
where the first equality follows from Equation \eqref{eqn:b_Lag_isot},
and the second equality holds because 
\[
d(\iota_L \times id_{[0,1]})_{(p,s)}\left(\frac{\partial}{\partial s}\right)
=\frac{\partial}{\partial s}
\]
for any $(p,s)\in L_{\gamma(t_0)}\times [0,1]$. We conclude that, for every $s\in [0,1]$, 
$\iota_L^*b_s=0$. 
The inclusion map $\iota_L: L_{\gamma(t_0)}\to  L$ is a homotopy equivalence, so $[b_s] =0 \in H^1(L)$.
\end{proof}

\begin{figure}[!ht]
    \centering
\resizebox{1.1\linewidth}{!}{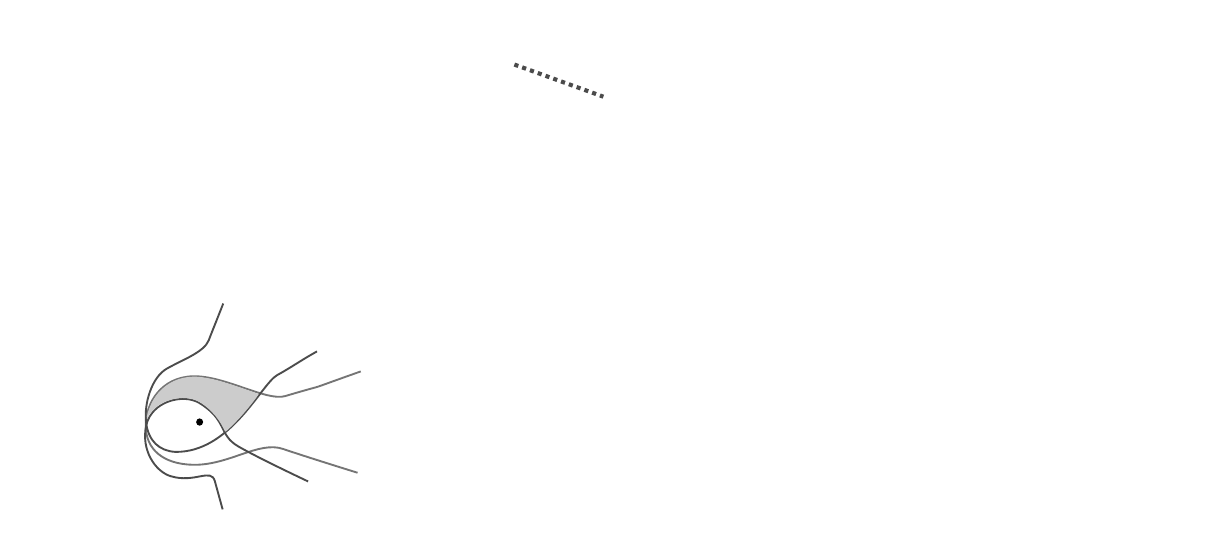}
    \caption{The three figures in the bottom illustrate the base of $(Y,v)$. The shaded triangle in the leftmost figure is the image under $v$ of a $J$-holomorphic triangle $u$ contributing to $\mu^2$. The Lagrangian isotopy $\psi$ deforms $L_0$ by {moving $v(L_0)$ to the left}. Theorem \ref{cor:triangle disc area} relates the symplectic area of $u$ to that of $u''$, which is a $J$-holomorphic triangle contained in $Y_{-\epsilon}$ with boundary on {three Lagrangians in the fiber}. The pyramid drawn at the top is a schematic illustration of the discs involved in the total space. The shaded triangle on the right {downward} face of the pyramid depicts {the image of} $u$. As we deform $L_0$ by the Lagrangian isotopy, $u$ deforms to $u'$, {whose image represents the union of $[u]$ and the top shaded quadrilateral in the pyramid}. The disc $u'$ is not $J$-holomorphic. The shaded triangle on the left side of the pyramid depicts a $J$-holomorphic disc $u''$ that is {relatively homologous to} $u'$.}
    \label{fig: prism_complex_paper5}
\end{figure}

Let $\cD(Y, (L_0,\ldots, L_k), (p_0,\ldots,p_k))$ be the set of continuous maps
\begin{equation}\label{eq: map u in D}
u: \Big(\bD, \partial \bD =\bigcup_{j=0}^k \partial_j\bD, \{ z_0,\ldots, z_k\} \Big) \lra \Big(Y, \bigcup_{j=0}^k L_j , \{ p_0,\ldots,p_k\} \Big)  
\end{equation}
from the unit disc $\bD$ to $Y$, where the images of the boundaries $u(\partial_j\bD)\subset L_j$, with corners at $u(z_j)=p_j$; see \cite[Definition 5.1]{ACLLb} for a more detailed definition.  For $J$ an almost complex structure on $Y$, let 
\[
\cM(Y, (L_0,\ldots, L_k), (p_0,\ldots, p_k); \beta, J)\\=\{u \in \cD(Y, (L_0,\ldots, L_k), (p_0,\ldots,p_k)) \mid [u] = \beta, \overline\partial_J(u)=0\}
\]
be the moduli space of $J$-holomorphic discs in $\cD(Y, (L_0,\ldots, L_k), (p_0,\ldots, p_k))$ that are in the homotopy class $\beta$.  The structure maps $\mu^k:CF(L_{k-1}, L_k)\otimes \cdots \otimes CF(L_0,  L_1)\to CF(L_0, L_k)$ of the Fukaya category are given by the count of $J$-holomorphic discs in the above moduli space, over all homotopy classes, weighted by a factor that depends on the symplectic area of each disc $u$; see \cite[Definition 5.13]{ACLLb}. {In an LG model, these discs are sections of $v$ \cite[Equation (17.2)]{seidel}. That is, we have the following result, under the assumption that the Lagrangians $\{L_j\}_{j=0}^k$ are in correct position as in Remark \ref{rmk:admissible flow}.}

\begin{lemma}\label{lem:mu_k} Let $(Y,\omega, v)$ be a symplectic Landau-Ginzburg model. Suppose that $J \in \text{End}(TY)$ is an almost complex structure compatible with $\omega$, in the sense that $\omega(\cdot,J \cdot)$ is a Riemannian metric and $J$ is an isometry. Furthermore, assume that $v:Y \to \bC$ is $(J,j)$-holomorphic where $j$ is a complex structure on $\mathbb{C}$. 

Take integer $k>0$. For complex structure $j_1$ on $\mathbb{D}$, if $u: (\bD, \dd \bD) \to (Y, \bigcup_{j=0}^k L_j)$ is a $(j_1,J)$-holomorphic curve counted in $\mu^k$, then $v \circ u$ is either constant or injective with image given by a $(k+1)$-gon.
\end{lemma}

\begin{corollary}\label{cor:section_or_fiber}
If $u: (\bD, \dd \bD) \to (Y, \bigcup_{j=0}^k L_j)$ is a $(j_1,J)$-holomorphic curve counted in $\mu^k$, it must be either contained in a fiber or it can be expressed as a section of $v$.
\end{corollary}

Thus $J$-holomorphic discs with fixed Lagrangian boundary in an LG model are either sections over the base or discs fully in the fiber. Note that $v\circ u$, as a holomorphic map, will satisfy a maximum principle in the base.

\begin{proof}[Proof of Lemma \ref{lem:mu_k}]
    We summarize the assumptions in the lemma with the following commutative diagram:
\begin{center}
    \begin{tikzcd}[column sep=large, row sep=large]
{(\mathbb{D},\partial \mathbb{D})} \arrow[rr, "u", "{(j_1,J)\text{-hol}}"'] \arrow[rrd, "v \circ u","{(j_1,j)\text{-hol}}"'] & & {(Y,\bigcup_{j=0}^k L_j)} \arrow[d, "v","{(J,j)\text{-hol}}"'] \\
& & {\mathbb{C}}
\end{tikzcd}
\end{center}

By the open mapping theorem for holomorphic maps, $v \circ u:(\mathbb{D},j_1) \to (\mathbb{C},j)$ is either constant or its image is an open subset of $\mathbb{C}$. If it is not constant, then  $v\circ u(\partial \bD)$ is a simple closed curve of winding number one around any point $a$ in its interior, as follows. On the boundary, $v\circ u$ maps to a counterclockwise configuration of $\gamma_{L_0}, \ldots, \gamma_{L_k}$ with transverse intersections. Transversality follows from $L_i \pitchfork L_j$; if $\gamma_{L_i}$ and $\gamma_{L_j}$ intersect non-transversely then consider the horizontal lifts of their tangent vectors at a base intersection point to one in the total space. The lifted vectors would be parallel and $L_i,L_j$ would intersect non-transversely, contradiction. As $k>0$, there is more than one Lagrangian in the boundary and the boundary must wind exactly once. This is because $v\circ u(\dd_j\bD) \subset \gamma_{L_j}$ by \cite[Definition 5.1, D5]{ACLLb} and $\gamma_{L_j} \pitchfork \gamma_{L_k}$ for $j \neq k$, so each $\gamma_{L_j}$ can only be traversed once. Therefore, the map is injective by the residue theorem and that the boundary of the curve winds once around any point $a$ in its interior. That is,
\begin{equation}
    \frac{1}{2\pi i} \oint_{\dd \mathbb{D}} \frac{d(v \circ u)}{(v \circ u)-a} = n(v \circ u(\dd \mathbb{D}),a) = 1 = \text{zeros}((v \circ u)-a)
\end{equation}
and so $v\circ u$ achieves the value $a$ exactly once.

\end{proof}

\begin{proof}[Proof of Corollary \ref{cor:section_or_fiber}]
If $v\circ u$ is constant at value $c$, then the disc given by the image of $u$ is contained in a fiber $Y_{c}$. Otherwise $v \circ u$ is injective by Lemma \ref{lem:mu_k} so there exists $z(s,t) \in \bD$ uniquely defined by
\begin{equation}
(v \circ u)(z(s,t))=R(s,t)
\end{equation}
for $(s,t) \in [0,1] \times [0,1]=I \times I$, where $R(s,t)$ parametrizes
\begin{equation}
    D:=(v \circ u)(\bD).
\end{equation}
Then we can define $u_D:D \to Y$  
\begin{equation}\label{eq:make_section}
    u_D(R(s,t)) = u(z(s,t))
\end{equation}
as a $((v\circ u)_*j_1,J)$-holomorphic section of $v$ if $u$ was a $(j_1,J)$-holomorphic curve.  
\end{proof}

For a disc $u \in \cD(Y, (L_0,\ldots, L_k), (p_0,\ldots, p_k))$, if we deform one of the boundary Lagrangians by a Lagrangian isotopy, we obtain a new disc $u'$ with the isotoped Lagrangian boundary.  We use the explicit construction of $u'$ given in \cite[Equation (6.11)]{ACLLb}, {see Figure \ref{fig:move_L2_paper5}}.  Note that $u'$ might not be $J$-holomorphic even if $u$ is.

\begin{corollary}\label{cor:endpts}
    For $u \in \cD(Y, (L_0,\ldots, L_k), (p_0,\ldots, p_k); \beta, J)$, let $u'$ be the deformation of $u$ obtained by performing a {fibered} Lagrangian isotopy $\psi$ on $L_m$, defined by \cite[Equation (6.11)]{ACLLb}. Then the difference in the symplectic areas of $u$ and $u'$ depends only on the isotopy $\psi$, the symplectic form $\omega$, and the endpoints of  $u(\dd_m \bD)$. 
\end{corollary}
\begin{proof} When deforming $L_m$ by an arbitrary Lagrangian isotopy, the difference in the symplectic areas of $u$ and $u'$ is provided in \cite[Theorem 6.10]{ACLLb}, specifically \cite[Equation (6.14)]{ACLLb}. That is,
\begin{equation}
    {\int_\bD u^*\omega = \int_\bD u'^* \omega + \int_{\dd_m \bD}(\dd_mu)^*\left(\int_0^1b_sds\right),}
\end{equation}
where $\int_0^1 b_sds$ is {exact by Lemma \ref{lem:Lag_isot_area}}.  That is, $b_s=dH_s$ and $\int_0^1b_sds=df$ for $f= \int_0^1 H_sds$ a smooth function on $L_m$ by \cite[Lemma 6.4]{ACLLb}.  Then by \cite[Remark 6.11]{ACLLb},
\begin{equation}
    {\int_\bD u^* \omega - \int_\bD u'^*\omega =f(p_{m+1})-f(p_m) }
\end{equation}
where $p_{m}$ and $p_{m+1}$ are the endpoints of $u(\dd_m\bD)$, and $b$, therefore $f$, only depends on $\psi$ and $\omega$. 
\end{proof}

\begin{remark}
    For the purpose of homological mirror symmetry, one generally needs to consider versions of Fukaya categories where the Lagrangians are equipped with extra structures such as a line bundle on the Lagrangian with a unitary connection. In that case, $\mu^k$ is given by a count of $J$-holomorphic discs weighted by a factor $\rho(u)$ that depends on both the symplectic area of each disc $u$ and the holonomy of the connection around the boundary of $u$.  Corollary \ref{cor:endpts} can be extended to say that the difference between the weights $\rho(u)$ and $\rho(u')$ also only depends on the endpoints of $u(\partial_m\bD)$, also due to \cite[Theorem 6.10, Lemma 6.4, and Remark 6.11]{ACLLb}.  The formula relating the weights given in \cite[Theorem 6.10]{ACLLb} depends only on the Lagrangian isotopy $\psi$ and the symplectic form $\omega$, not on the connection.
\end{remark}

The primary application of the following theorem is to assist in proving examples of the homological mirror symmetry conjecture. One way to prove evidence of the conjecture is to perform calculations on both sides of the mirror correspondence and show they are equal. One such step is composition $\mu^2$ in the Fukaya category of the symplectic manifold (and show it mirrors composition in the mirror category on the mirror manifold). This structure map counts solutions to a PDE with boundary conditions specified by Lagrangians. There is ambiguity in how fibered Lagrangians are drawn because they depend on the curves chosen in the base. Thus, the following theorem can be used to compute disc areas appearing in $\mu^2$ in a symplectic Landau-Ginzburg model by decomposing into a disc area appearing in $\mu^2$ in the reference fiber, plus a base contribution. One application of this theorem is \cite[Equation (3.88)]{CV}. Another application will appear in forthcoming work on homological mirror symmetry for theta divisors \cite{ACLL2}.

\begin{theorem}\label{cor:triangle disc area}
  Let $(Y,\omega, v)$ be a symplectic Landau-Ginzburg model. Suppose that $J \in \text{End}(TY)$ is an almost complex structure compatible with $\omega$, in the sense that $\omega(\cdot,J \cdot)$ is a Riemannian metric and $J$ is an isometry. Furthermore, assume that $v:Y \to \bC$ is $(J,j)$-holomorphic where $j$ is a complex structure on $\mathbb{C}$. Let 
   $u \in \cM(Y, (L_0,L_1, L_2), (p_0,p_1, p_2); \beta, J)$ be a $J$-holomorphic triangle counted in $HF(L_{1},L_2) \otimes HF(L_0, L_{1}) \to HF(L_0,L_2)$, where $L_0, L_1, L_2$ are parallel fibered Lagrangians with $\gamma_{L_0}(0)=\gamma_{L_1}(0) = \gamma_{L_2}(0)$.  
  
   Then the area $\int_\bD u^* \omega$ of $u$ is a sum of base and fiber contributions. The fiber contribution is the area of a curve counted in the composition in the fiber $HF(\ell_1',\ell_2') \otimes HF(\ell_0',\ell_1') \to HF(\ell_0',\ell_2')$ for fiber Lagrangians $\ell_j'\in Y_{\gamma_{L_j}(0)}$ depending on the singularities of $v$.
\end{theorem}

\begin{proof}
\begin{figure}
    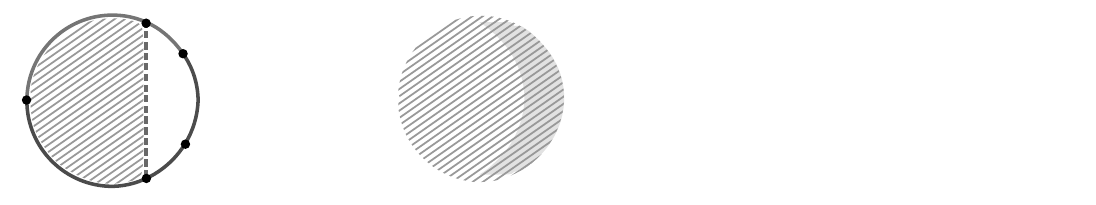
        \caption[LoF entry]{The leftmost disc in this figure is the domain for maps $u$  and $u'$, with $u$ in the form of Equation \eqref{eq: map u in D} with $k+1=3$. As we deform $L_0$ by a Lagrangian isotopy $\psi$, the map $u$ deforms to $u':(\bD, \partial'_0\bD\cup \partial'_1\bD\cup \partial'_2\bD, \{z_0', z_1', z_2\})\to (Y, \psi_1(L_0)\cup L_1 \cup L_2, \{\psi_1(p_0), \psi_1(p_1), p_2\})$. 
        
       \hspace{1em} We use the piecewise smooth construction for $u'$ given in \cite[Equation (6.11)]{ACLLb}.  To very briefly summarize, this construction involves first deforming the domain disc $\bD$ using  maps $\phi_1$ and $\phi_2$.  The map $\phi_1$ sends the region shaded by lines in the leftmost figure to the entire $\bD$ in the middle figure.  The map $\phi_2$ is orientation reversing, and it sends the solidly shaded region in the leftmost figure to $\partial_0\bD\times [-1,0]$ as shown in the middle figure.  Both maps agree on
    the line segment $\overline{z_0z_1}$, which is the dotted vertical line in the leftmost figure, so $\phi_1\cup_{\overline{z_0z_1}}\phi_2$ is a continuous map.  The map $u'$ is then defined by composing $\phi_1$ by $u$ and composing $\phi_2$ by $\psi_s\circ \partial_0 u$.    

    \hspace{1em} In the case of Theorem \ref{cor:triangle disc area}, the image of $u'$ is the union of the faces on the right side and top of the pyramid shown in Figure \ref{fig: prism_complex_paper5} (the illustrations in the first two images in Figure \ref{fig:move_L2_paper5} are the domains). The rightmost picture in this Figure \ref{fig:move_L2_paper5} shows the image of $u'$ under $v$.}
        \label{fig:move_L2_paper5}
    \end{figure}  

    Let $u'\in \cD(Y, (\psi_1(L_0),L_1, L_2), (\psi_1(p_0), \psi_1(p_1), p_2))$ be the disc obtained by deforming $L_0$ using the parallel fibered Lagrangian isotopy $\psi$ in Equation \eqref{eq:parallel isotopy} such that $\gamma_{L_0,s}$ moves to $c=v(p_2)$. That is, the isotopy occurs as $s$ increases from $s=0$ to $s=1$, though, visually in the base in Figure \ref{fig: prism_complex_paper5}, the arc in $v(L_0)$ is moving from right to left. The map $u'$ is defined by keeping $u$ and then ``stretching" it further as we move the Lagrangian \cite[Equation (6.11)]{ACLLb}. Then we show that $[u']$ admits a relatively homologous representative $u''$ completely contained in the fiber $Y_{c}$, with $[u']=[u'']\in H_2(Y,\psi_1(L_0) \cup L_1 \cup L_2)$.  We do this by constructing a contractible singular 3-chain $Q$ whose relative boundary is $[u']-[u'']$. It is a continuous map from the unit cube to $Y$, with image depicted in Figure \ref{fig: prism_complex_paper5}. The part of the boundary of $Q$ given by the quadrilateral on the top face in Figure \ref{fig: prism_complex_paper5} is an example of $P$ in \cite[Equation (6.11)]{ACLLb}, for LG models. 
    
    First we define the region in the base that $Q$ will cover. Define $R(s,t): I \times I \to \bC$ so that the top of the square $I \times \{1\}$ maps to the $v(L_2)$ side of the triangle in the base, then going around counterclockwise the left of the square $\{0\} \times I$ maps to $v(L_0)$, the bottom of the square $I \times \{0\}$ maps to $v(L_1)$, and the right of the square $\{1\} \times I$ maps to $c$. We can reparametrize, with parameter $t_L$, the base curves defining the Lagrangians to match $R(s,t)$ along the boundary. And $R: (0,1) \times (0,1) \to (v\circ u)(\text{int}(\bD))$ is a homeomorphism on the interior. Therefore,
\begin{equation}\label{eq:match_params}
\begin{aligned}
    R(s,1) & = \gamma_{L_{2}}(t_{L_{2}}(s)) \\
    R(0,t) &= \gamma_{L_{0}}(t_{L_{0}}(t))  \\
    R(s,0) &= \gamma_{L_{1}}(t_{L_{1}}(s)) \\
    R(1,t) &= c.
    \end{aligned}
\end{equation}
We denote the isotopy in the base $v(\psi_s(L_0))$ by $\gamma_{L_0,s}$ so that
\begin{equation}
    h(t,s)=R(s,t) = \gamma_{L_0,s}(t_{L_0,s}(t)), \qquad (s,t) \in I \times I.
\end{equation}
    
Now we define $Q$. Recall $\Phi_{\gamma(0)  \xrightarrow[]{\gamma} \gamma(1)}$ denotes parallel transport along a path $\gamma$ between the fibers over $\gamma(0)\to \gamma(1)$. Define the singular 3-chain $Q$ to be $u_D$ from Equation \eqref{eq:make_section} on the bottom face $I \times I \times \{0\}$, the continuous map 
\begin{equation}\label{eq:Jhol_tri_fiber}
u''(z(s,t)):=\Phi_{\gamma_{L_0}(0)=\gamma_{L_0,1}(0)\xrightarrow[]{\gamma_{L_0,1}} R(1,t)}\Phi_{R(s,t) \xrightarrow[]{\gamma_{L_0,s}} \gamma_{L_0}(0)}\circ u_D(R(s,t))    
\end{equation}
to the fiber over $c$ on the top face $I \times I \times \{1\}$, and interpolate in between:
    \begin{equation}
        \begin{aligned}
              Q:  I \times I \times I & \to Y\\
(s,t,e) & \mapsto \begin{cases}
    u_D(R(s,t)), & s \geq e\\
\Phi_{\gamma_{L_0}(0) \xrightarrow[]{\gamma_{L_0,e}} R(e,t)}     \Phi_{R(s,t) \xrightarrow[]{\gamma_{L_0,s}} \gamma_{L_0}(0)} \circ u_D(R(s,t)),& s \leq e
\end{cases}.
        \end{aligned}
    \end{equation} 

   Informally, the image looks like peeling a sticker off of $u_D$ and towards the fiber over $c$. When $e=0$, $Q|_{I \times I \times \{0\}}=u_D$ maps to the image of the original curve $u$. When $e=1$, $Q|_{I \times I \times \{1\}}=u''$ maps to the fiber over $c$. In between, for a fixed $e$, $Q|_{I \times I \times \{e\}}$ maps to the original curve $u$ for $s \in [e,1]$. As $s$ decreases from $e$ to 0, the image looks like peeling back a sticker. The boundary of $Q$ is depicted in Figure \ref{fig: prism_complex_paper5} as follows:
    
\begin{equation}\label{eq:boundaryofQ}
    \begin{aligned}
        e=0, \quad 0 \leq s,t \leq 1 \quad & \mapsto [u] \text{ (downward-facing triangle) }\\
        e=1, \quad 0 \leq s, t \leq 1 \quad & \mapsto [u''] \text{ (fiber triangle) }\\
        s=0, \quad 0 \leq e, t \leq 1 \quad & \mapsto [u'] - [u] \text{ (top quadrilateral obtained from the Lagrangian isotopy) }\\
        s=1, \quad 0 \leq e, t \leq 1 \quad & \mapsto u_D(c) \text{ (constant at bottom vertex) }\\
        t=0, 1, \quad 0 \leq s,e \leq 1 \quad & \mapsto \text{image} \subset L_{1} \cup L_{2} \text{ (back and front triangles) }.
    \end{aligned}
\end{equation}

By Equation \eqref{eq:Jhol_tri_fiber}, $u''(\bD) \subset Y_c$ is in the fiber over $c$ with boundary on the following Lagrangians. Recall that $L_j$ is obtained by parallel transporting $\ell_j \in Y_{\gamma_{L_j}(0)}$ along $\gamma_{L_j}$. Under the Lagrangian isotopy, we trace backwards along this path when $j=1,2$. So these two Lagrangians are fiber Lagrangians, parallel transported from $\gamma_{L_j}(0)$ to $c$:
\begin{equation}
  \Phi_{ \gamma_{L_j}(t_{L_j}(s)) \to \gamma_{L_j}(t_{L_j}(1))} \circ \Phi_{\gamma_{L_j}(0)\to \gamma_{L_j}(t_{L_j}(s))}(\ell_j) = \Phi_{\gamma_{L_j}(0) \to \gamma_{L_j}(t_{L_j}(1))} (\ell_j), \qquad j=1,2.
\end{equation}
From $s=0$ to $s=1$, the side on $L_0$ parallel transports to 
\begin{equation}
    \Phi_{\gamma_{L_0}(0) \xrightarrow[]{\gamma_{L_0,1}} R(1,t)}     \Phi_{R(0,t) \xrightarrow[]{\gamma_{L_0,0}} \gamma_{L_0}(0)}(\Phi_{ \gamma_{L_0}(0)  \xrightarrow[]{\gamma_{L_0,0}} R(0,t)}\ell_0)=\Phi_{\gamma_{L_0}(0) \xrightarrow[]{\gamma_{L_0,1}} R(1,t)} (\ell_0).
\end{equation}
Thus $u''$ is in the fiber bounded by fiber Lagrangians:
\begin{equation}
    [u''] \in H_2(Y_c, \Phi_{\gamma_{L_0}(0) \xrightarrow[]{\gamma_{L_0,1}} R(1,t)} (\ell_0) \cup \Phi_{\gamma_{L_1}(0) \to \gamma_{L_1}(t_{L_1}(1))} (\ell_1) \cup \Phi_{\gamma_{L_2}(0) \to \gamma_{L_2}(t_{L_2}(1))} (\ell_2))
\end{equation}
and by Equation \eqref{eq:boundaryofQ}
\begin{equation}
\dd Q = 0 = -[u]+[u'']-([u']-[u])+(\text{class in Lagrangians}) 
\end{equation}
which implies, relative the Lagrangians,
\begin{equation}\label{eq:homolog_tris}
    [u']=[u''] \in H_2\left(Y, \psi_1(L_0) \cup L_{1} \cup L_{2} \right).
\end{equation}

By Corollary \ref{cor:endpts}
\begin{equation}
    \int_\bD u^* \omega  = \int_\bD (u')^*\omega + \int_{\dd_0 \bD}(\dd_0 u)^* \theta_\psi = \int_\bD (u')^*\omega + f(p_1)-f(p_0)
\end{equation}
where recall
\begin{equation}
    \psi^*\omega = b \wedge ds, \qquad \theta_\psi = \int_0^1 b_s ds=df.
\end{equation}
We know $\int_{u'(\bD)} \omega = \int_{u''(\bD)} \omega$ by Equation \eqref{eq:homolog_tris}. Therefore,
\begin{equation}
\int_\bD u^*\omega = \int_\bD (u'')^*\omega +  f(p_1)-f(p_0).
\end{equation}
The term $\int_\bD (u'')^*\omega$ is a contribution only in the fiber. The term $f(p_1)-f(p_0)$ depends only on the isotopy $\gamma_{L_0,s}$ in the base, where
$$
[\gamma_{L_0,I}(t_{L_0,I}(I))]=v_*[\beta] \in H_2(v \circ u(\bD), \gamma_{L_0}(t_{L_0}(I)) \cup \gamma_{L_1}(t_{L_1}(I)) \cup \gamma_{L_2}(t_{L_2}(I))).
$$
Note that $\beta$ is a relative homotopy class, but for area it suffices to consider its relative homology.

To compute in the reference fiber over $\gamma_{L_0}(0) = \gamma_{L_1}(0) = \gamma_{L_2}(0)$, we apply the symplectomorphism induced by moving along a path from $c$ to $\gamma_{L_0}(0)$ given by $\Phi_{c \xrightarrow[]{\gamma_{L_0,1}} \gamma_{L_0}(0)}$. Let
\begin{equation}
\begin{aligned}
    \Phi_1 & := \Phi_{c \xrightarrow[]{\gamma_{L_0,1}} \gamma_{L_0}(0)}\Phi_{\gamma_{L_1}(0) \to \gamma_{L_1}(t_{L_1}(1))}\\
    \Phi_2 & :=\Phi_{c \xrightarrow[]{\gamma_{L_0,1}} \gamma_{L_0}(0)} \Phi_{\gamma_{L_2}(0) \to \gamma_{L_2}(t_{L_2}(1))} .
\end{aligned}
\end{equation} 
Then the disc $\Phi_{c \xrightarrow[]{\gamma_{L_0,1}} \gamma_{L_0}(0)} \circ u''$ represents a class in 
$$
H_2(Y_{\gamma_{L_0}(0)},\ell_0 \cup \Phi_1(\ell_1) \cup \Phi_2(\ell_2)).
$$ 
To compute Floer theory in the reference fiber, which is invariant under Hamiltonian equivalence of Lagrangians, one needs to determine each monodromy $\Phi_1$ and $\Phi_2$. The resulting Lagrangians are $\ell_0'=\ell_0, \ell_1'=\Phi_1(\ell_1)$, and $\ell_2'=\Phi_2(\ell_2)$.  Note that one may equivalently go from $c$ to $\gamma_{L_0}(0) = \gamma_{L_1}(0) = \gamma_{L_2}(0)$ along $\gamma_{L_1}$ (or $\gamma_{L_2}$) and determine the corresponding monodromies on $\ell_0$ and $\ell_2$ (or $\ell_0$ and $\ell_1$), respectively.
\end{proof}

\begin{remark}Note that once we perform the Lagrangian isotopy as in Figure \ref{fig: prism_complex_paper5}, the three intersection points $p_0\in L_0\cap L_2$, $p_1 \in L_0 \cap L_1$, and $p_2 \in L_1 \cap L_2$ have all moved to the same fiber $Y_{c}$. The only $J$-holomorphic discs in $\cM(Y, (\psi_1(L_0),L_1, L_2), (\psi_1(p_0),\psi_1(p_1), p_2); \beta', J)$  that are relatively homologous to $u'$ can be shown to be contained in the fiber over $c$ using the open mapping theorem and maximum modulus principle. The boundary of the domain depicted in Figure \ref{fig:move_L2_paper5} maps under $u'$ and then $v$ as follows
    \begin{equation}
  \dd \bD \xrightarrow[]{u'}   {\psi_1(u(\dd_0\bD))} \cup {\bigcup_{s=0}^1 \psi_s(u(z_1))}\cup {u(\dd_1\bD)}  \cup {u(\dd_2 \bD)}\cup {\bigcup_{s=0}^1 \psi_s(u(z_0))} \xrightarrow[]{v} c \cup {\gamma_{L_1}(t_{L_1}([0,1]))} \cup {\gamma_{L_2}(t_{L_2}([0,1]))}.
    \end{equation}
This consists of two line segments, which therefore cannot bound a nontrivial disc. In the proof of Theorem \ref{cor:triangle disc area}, we constructed a specific relatively homologous curve in the fiber that can be used in computations.
    
\end{remark}

\begin{remark}
    We expect a more general statement can be made for $\mu^k$. In this case, we can move all intersection points to the fiber over $c=v(p_k)$ with $k-1$ Lagrangian isotopies $\{h^j(t,s)\}_{j=0}^{k-2}$. The isotopy on the boundary $h^0(t,0)$ moves $v(p_0) \in v(L_0 \cap L_k)$ clockwise along $v(L_k)$ to $v(p_k)$. For $j=1,\ldots,k-3$ the $j$th projected intersection point $v(p_j)\in v(L_j \cap L_{j-1})$  is moved clockwise along $h^{j-1} \circ \ldots \circ h^0 \circ v(L_j)$. That is, we move each intersection point in the base to $c$, isotoping one Lagrangian after another so that the intersection point in the base traces along the previously moved projected Lagrangian, to $c$. At this point, we are in the case of Theorem \ref{cor:triangle disc area} in the base and need one more isotopy to move the last two intersection points to lie over $v(p_k)$.
\end{remark}

\section{Grading in a symplectic Landau-Ginzburg model}\label{sec:Maslov_fibrn}

To prove that the $\bZ$-grading on morphisms splits into  fiber and base contributions, we show that such a splitting occurs for each of the concepts needed to define the grading. That is, in \textsection \ref{sec:fr_bdle_paper5} we show that the structure groups for symplectic and unitary frame bundles on the symplectic Landau-Ginzburg model reduce to block-diagonal matrices, in particular the Lagrangian Grassmannian splits into base and fiber parts in \textsection \ref{sec:LGr_paper5}. Hence, quadratic volume forms and corresponding squared phase maps split in \textsection \ref{sec:quadratic}. Since these can be used to define a grading on a Lagrangian, we have a base and fiber splitting of Lagrangian gradings in \textsection \ref{sec:grade_Lag_paper5}. Lastly, as the $\bZ$-grading on intersection points can be defined from the Lagrangian gradings, we have a splitting on the $\bZ$-grading of the morphisms in \textsection \ref{sec:grade_mors_paper5}.

\subsection{Framed bundles}\label{sec:fr_bdle_paper5}
We still have a decomposition into $\Ver$ and $\Hor$ on a larger space containing the symplectic fibration $Y_C \to C$. Let $\mathrm{crit}(v)\subset Y$ be the set of critical points of $v$, and let $Y':= Y\setminus \mathrm{crit}(v)$. Then $TY' = TY|_{Y'}$ is the direct sum
    of two subbundles:
    $$
   TY' = \Ver \oplus \Hor.
$$

\begin{remark}
    The horizontal vector subbundle $\Hor$ can be trivialized by Hamiltonian vector fields $\{X_{\re(v)}$, $X_{\im(v)}\}$ defined by $\iota_{X_{\re(v)}} \omega = d \re(v)$ and $\iota_{X_{\im(v)}} \omega = d \im(v)$. This realizes an explicit isomorphism $\Hor \stackrel{\cong}{\lra} Y'\times \bR^2$ of rank 2 real vector bundles over $Y'$. 
\end{remark}

\subsubsection{Symplectic frame bundles}
In the proof of Lemma \ref{lem:Lag_isot_area},  $\omega^\Ver$ and $\omega^\Hor$ are sections of
$\Lambda^2 \Ver^*$ and $\Lambda^2 \Hor^*$, respectively.
The restriction of the symplectic vector bundle $(TY,\omega)$ to $Y'$ is an $\omega$-orthogonal direct sum of two symplectic vector bundles:
\begin{equation}\label{eqn:TY-omega}
(TY', \omega|_{Y'})= (TY,\omega)\big|_{Y'} = (\Ver, \omega^\Ver) \oplus (\Hor, \omega^\Hor).
\end{equation}

The symplectic frame bundle $F(\Ver, \omega^\Ver)$ 
(resp. $F(\Hor,\omega^\Hor)$) is
a principal $\Sp(2n,\bR)$-bundle (resp. $\Sp(2,\bR)$-bundle) over $Y'$. The fiber product
$$
F(\omega^\Ver\oplus \omega^\Hor):= F(\Ver,\omega^\Ver)\times_{Y'} F(\Hor, \omega^\Hor)
$$
is a principal $H$-bundle over $Y'$ where $H=\Sp(2n,\bR)\times \Sp(2,\bR)$. The fiber of $F(\omega^\Ver \oplus \omega^\Hor)$ over a point
$p\in Y'$ is
$$
F(\omega^\Ver\oplus \omega^\Hor)_p = \Big\{ (e_1,\ldots, e_{2n+2}) \ : \ \begin{array}{ll} (e_1,\ldots,e_{2n})\text{ is a symplectic basis of  }(\Ver_p, \omega^\Ver(p))  \\
(e_{2n+1}, e_{2n+2})\textup{ is a symplectic basis of }(\Hor_p,\omega^\Hor(p))  \end{array}
\Big\}
$$
The symplectic frame bundle $F(TY,\omega)$ is a principal
$\Sp(2n+2,\bR)$-bundle over $Y$; its fiber over a point $p\in Y$ is
$$
F(TY,\omega)_p =\{ (e_1,\ldots, e_{2n+2}) \ : \ (e_1,\ldots, e_{2n+2}) \textup{ is a symplectic basis of }(T_pY, \omega(p))\}. 
$$
We have an inclusion $F(\omega^\Ver\oplus \omega^\Hor)\subset F(TY,\omega)|_{Y'} = F(TY', \omega|_{Y'})$. The structure group
of the principal $\Sp(2n+2, \bR)$-bundle $F(TY',\omega|_{Y'})$ can be reduced to the subgroup $H=\Sp(2n,\bR)\times \Sp(2,\bR)$:
$$
F(TY', \omega|_{Y'} ) = 
F(\omega^\Ver\oplus \omega^\Hor) \times_H \Sp(2n+2,\bR).
$$

\subsubsection{Unitary frame bundles}\label{sec:unitary-frame}
Let $z=x+iy$ be the complex coordinate on $\bC$. We equip $\bC$ with the standard complex structure $\bj:T\bC\to T\bC$ given by multiplication by $i$, which is compatible with the standard symplectic structure $$\omega_{\bC} =dx\wedge dy =\frac{i}{2}dz\wedge d\bar{z}.$$
We now assume $v: Y\to \bC$ is $J$-holomorphic where $J: TY\to TY$ is an almost complex structure compatible with $\omega$.  Then $J$ preserves $\Ver =\ker(dv|_{Y'})$ and $\Hor=(\ker(dv|_{Y'}))^{\perp_\omega}$ because $\bj\circ dv = dv\circ J$; its restrictions $J^\Ver:\Ver\to \Ver$ and $J^\Hor:\Hor\to \Hor$
are compatible with $\omega^\Ver$ and $\omega^\Hor$, respectively. 
The restriction of the $C^\infty$ complex vector bundle of $T_{\bC}Y = (TY,J)$ to $Y'$ is a direct sum of two $C^\infty$ complex subbundles:
\begin{equation}\label{eqn:TY-J}
T_\bC Y' = \Ver_{\bC} \oplus \Hor_{\bC}.
\end{equation}
where $\Ver_{\bC} = (\Ver, J^{\Ver})$ and $\Hor_{\bC}= (\Hor,J^\Hor)$. 
\begin{remark}
We have an isomorphism
$\Hor_\bC\to (v|_{Y'})^*T\bC$ of $C^\infty$ complex line bundles over $Y'$. However, this is not an isomorphism
of symplectic vector bundles in general. More precisely,
$$
\left. (v^*\omega_{\bC})\right|_{\Hor} = |dv|_{J,\omega}^2 \omega^\Hor.
$$
\end{remark}
Combining \eqref{eqn:TY-omega} and \eqref{eqn:TY-J}, we obtain the following decomposition as a direct sum of two Hermitian vector bundles:
\begin{equation}\label{eqn:TY-hermitian}
(T_\bC Y', \omega|_{Y'}) = 
(\Ver_\bC,\omega^\Ver)\oplus
(\Hor_\bC, \omega^\Hor).
\end{equation}
The unitary frame bundle $F(\Ver_\bC, \omega^\Ver)$ 
(resp. $F(\Hor_\bC, \omega^\Hor)$) is
a principal $U(n)$-bundle (resp. $U(1)$-bundle) over $Y'$. The fiber product
$$
F\left( (\omega^\Ver, J^\Ver) \oplus (\omega^\Hor, J^\Hor) \right) := F(\Ver_\bC,\omega^\Ver)\times_{Y'} F(\Hor_\bC, \omega^\Hor)
$$
is a principal $U(n)\times U(1)$-bundle over $Y'$.
The fiber of $F( (\omega^\Ver, J^\Ver) \oplus (\omega^\Hor, J^\Hor) )$ over a point
$p\in Y'$ is
\begin{eqnarray*}
&& F( (\omega^\Ver, J^\Ver) \oplus (\omega^\Hor, J^\Hor) )_p  \\
&=& \Big\{ (e_1,\ldots, e_{n+1}) \ : \ \begin{array}{ll} (e_1,\ldots,e_n)\text{ is an ordered orthonormal $\bC$-basis of }( (\Ver_{\bC})_p \simeq \bC^n, \omega^\Ver(p))  \\
e_{n+1} \textup{ is a unit vector in  }( (\Hor_{\bC})_p \simeq \bC,\omega^\Hor(p))  \end{array}
\Big\}
\end{eqnarray*}
The unitary frame bundle $F(T_{\bC}Y,\omega)$ is a principal
$U(n+1)$-bundle over $Y$; its fiber over a point $p\in Y$ is
$$
F(T_{\bC}Y,\omega)_p =\{ (e_1,\ldots, e_{n+1}) \ : \ (e_1,\ldots, e_{n+1}) \textup{ is an orthonormal $\bC$-basis of }( (T_{\bC}Y)_p\simeq \bC^{n+1}, \omega(p))\}. 
$$
We have an inclusion $F( (\omega^\Ver, J^\Ver) \oplus (\omega^\Hor, J^\Hor) )\subset F(T_{\bC} Y',\omega|_{Y'})$. The structure group
of the principal $U(n+1)$-bundle $F(T_{\bC}Y',\omega|_{Y'})$ can be reduced to the subgroup $U(n)\times U(1)$:
$$
F(T_{\bC}Y', \omega|_{Y'}) = 
F \left( (\omega^\Ver, J^\Ver)\oplus (\omega^\Hor,  J^\Hor) \right) \times_{U(n)\times U(1)} 
U(n+1).
$$

\subsection{Lagrangian Grassmannian bundles}\label{sec:LGr_paper5}
Over $Y$, we have a Lagrangian Grassmannian bundle $LGr(TY,\omega)= F(T_\bC Y, \omega)/O(n+1)$ and a map
$$
\phi_Y: LGr(TY,\omega)\to  F\left(\det(T_{\bC}Y)^{\otimes 2} \right)
$$
where $F(\det(T_{\bC}Y^{\otimes 2}))$ is the unitary frame bundle of the Hermitian line bundle $\det(T_{\bC}Y)^{\otimes 2}$, and in particular it is 
a principal $U(1)$-bundle over $Y$.

Over $Y'$, we have Lagrangian Grassmannian bundles
$$
LGr(\Ver, \omega^\Ver) = F(\Ver_{\bC}, \omega^{\Ver})/O(n), \quad
LGr(\Hor, \omega^{\Hor}) = F(\Hor_{\bC}, \omega^\Hor)/O(1).
$$

\subsection{Quadratic complex volume forms and squared phase maps}\label{sec:quadratic}

\paragraph{\bf Assumptions.} In this section, we  assume $(Y, \omega, J)$ is a K\"ahler manifold of complex dimension $n+1 = \dim_{\bC} Y$, where $\omega$ is the symplectic structure, $J$ is the complex structure, and $v: Y \to \bC$ is a holomorphic function such that $\critv(v)$ is a finite set. 
Again, let $C =\bC\backslash \{ \critv(v)\}$ be the set of regular values of the superpotential $v$. For any $c\in \bC$, let $Y_c=v^{-1}(c)$, and let $Y'_c = v^{-1}(c)\cap Y'$.  For any $c \in C$, $Y_c=Y'_c$ is a closed complex submanifold of $Y$ of dimension $n$, and $\omega_c = \omega|_{Y_c} \in \Omega^2(Y_c)$ is a K\"ahler form on $Y_c$. Then $Y_C= v^{-1}(C) \to C$ is again a symplectic fibration in the sense of \cite{symp_intro}, and we further assume that there is a nowhere vanishing holomorphic $(n+1)$-form $\Omega$ on $Y$; $\Omega$ is also called a {\em holomorphic volume form}. In particular, $Y$ is a Calabi-Yau manifold. For any $c \in \critv(v)$, 
$Y'_c$ is a locally closed complex submanifold of $Y$ and a closed complex submanifold of $Y'$, 
and $\omega|_{Y'_c} \in \Omega^2(Y'_c)$ is a K\"{a}hler form on $Y'_c$. 

\begin{remark} One of the main references of this section is 
Section 15 of Seidel's book \cite{seidel}, where he considers
$(I_E, I_B)$-holomorphic maps $\pi:E\to B$ between exact symplectic manifolds with corners, i.e., 
$$
d\pi\circ I_E = I_B \circ d\pi
$$
where $I_E:TE\to TE$ (resp. $I_B:TB\to TB$) is an almost complex structure  on $E$ (resp. $B$) compatible with the exact symplectic form $\omega_E$ (resp. $\omega_B$) on $E$ (resp. $B$). {This allows for a more general treatment, dropping the K\"ahler assumption, see also \cite[\textsection 4.5]{ACLLb}. As the focus of this article is on applicability in computations, we show how a quadratic volume form is often found in practice. Many examples in homological mirror symmetry are Calabi-Yau. Therefore they admit a holomorphic volume form $\Omega$ defining a quadratic complex volume form, as follows. So our set-up is a special case of the general setting. }
\end{remark}

 The K\"ahler metric on $Y$ determines a Hermitian metric on the holomorphic line bundle $K_Y= \det(T^*_\bC Y)= \bigwedge^{n+1} T^*_\bC Y$. Let $|\Omega|$ denote the norm of the holomorphic volume form $\Omega$ with respect to the K\"ahler metric. Then $|\Omega|$ is a positive smooth function on $Y$ (we do not assume $\Omega$ is parallel, so $|\Omega|$ is not necessarily a constant), $\Omega/|\Omega|$ is a unitary frame of $\det(T^*_\bC Y)$, and 
$$
\Theta := (\Omega/|\Omega|)^{\otimes 2}
$$ 
is a unitary frame of $\det(T^*_\bC Y)^{\otimes 2}$ (called a {\em quadratic complex volume form} in \cite{seidel}).  
\begin{definition}
We define $\alpha_{\Theta}: LGr(TY,\omega)\to U(1)$ (called a {\em squared phase map} in \cite{seidel}) by 
\begin{equation}\label{eqn:alpha}
\alpha_{\Theta}(p, V) = \frac{\Omega(e_1,\ldots, e_{n+1})^2}{ |\Omega(e_1,\ldots, e_{n+1})|^2}
\end{equation}
where $p\in Y$, $V\cong \bR^{n+1}$ is a linear Lagrangian subspace of $(T_pY, \omega(p))$, and $(e_1,\ldots,e_{n+1})$ is {an $\bR$-basis} of $V$. 
\end{definition}
Note that $\alpha_{\Theta}(p,V)$ is well-defined because (i) $\Omega(e_1,\ldots, e_{n+1})\in \bC$ is non-zero since $V$ is Lagrangian, and (ii) if $(e'_1,\ldots, e'_{n+1})$ is {another $\bR$-basis} of $V$ then $\Omega(e_1',\ldots, e_{n+1}') =\lambda\Omega(e_1,\ldots, e_{n+1})$ for some nonzero $\lambda \in \bR$.

For any $ c \in \bC$, we define a nowhere vanishing holomorphic $n$-form $\Omega_c$ on $Y'_c$ as follows. 
Note that $dv|_{Y'}$ is a nowhere vanishing holomorphic 1-form on $Y'$. Given any point $p\in Y'_c$, there
exists an open neighborhood $U$ of $p$ in $Y'$ and local holomorphic coordinates $(z_1,\ldots, z_{n+1})$ on 
$U$ such that $z_{n+1}=v$. Then $(z_1,\ldots, z_n)$ are local holomorphic coordinates on 
$U\cap Y'_c$.  In terms of local holomorphic coordinates  $(z_1,\ldots, z_{n+1})$ on $U$, 
$$
\Omega = a(z_1,\ldots, z_{n+1}) dz_1\wedge \cdots \wedge dz_{n+1}
$$
where $a(z_1,\ldots, z_{n+1})$ is a nowhere vanishing holomorphic function on  $U$. In terms of local holomorphic coordinates $(z_1,\ldots, z_n)$ on $U\cap Y'_c$ (which is an open neighborhood of $p$ in $Y'_c$), 
$$
\Omega_c = a(z_1,\ldots, z_n, c) dz_1 \wedge \cdots \wedge dz_n.
$$
The holomorphic $n$-form $\Omega_c$ is known as the Poincar\'{e} residue of the meromorphic $(n+1)$-form $\frac{\Omega}{v-c}$ along the hypersurface $Y'_c$ in $Y'$, as in \cite[p 147]{gh}. In particular, $Y'_c$ is a Calabi-Yau manifold.  
\begin{remark}
    If $c\in C$ and $Y_c= Y'_c$ is compact  then $H^0(Y_c, K_{Y_c})=\bC \Omega_c$, i.e., any holomorphic $(n,0)$-form on $Y_c$ is a constant multiple of $\Omega_c$.
    \end{remark} 

Similarly to the total space, for each $c \in \bC,$ 
$$
\Theta_c:=(\Omega_c/|\Omega_c|)^{\otimes 2}
$$ 
is a unitary frame of $\det(T_{\bC}^*Y'_c)^{\otimes 2}$.  
We use $\Theta_c$ to define $\alpha_{\Theta_c}: LGr(TY'_c, \omega|_{Y'_c})\to U(1)$. 

For every $c\in \bC$, the inclusion map $\iota_c: Y'_c \to Y$ is holomorphic, and we have an isomorphism
$\iota_c^* \Ver_\bC \to T_{\bC}Y'_c$ of holomorphic vector bundles over $Y'_c$. There is a holomorphic section 
$\Omega^\Ver$ of $\det(\Ver_\bC^*)$ such that $\iota_c^*\Omega^\Ver =\Omega_c \in H^0\left(Y'_c, K_{Y'_c} \right)$ for all $c\in \bC$. Then
$$
\Theta^\Ver:= (\Omega^\Ver/|\Omega^\Ver|)^{\otimes 2}
$$
is a unitary frame of $\det(\Ver_\bC^*)^{\otimes 2}$ (called a {\em relative quadratic complex volume form} in \cite{seidel})
and $\iota_c^*\Theta^\Ver=\Theta_c$ for all $c\in \bC$. 
\begin{definition} We define $\alpha_{\Theta^\Ver}: LGr(\Ver, \omega^\Ver)  \to U(1)$ (called a {\em relative squared phase map} in \cite{seidel}) by
\begin{equation}\label{eqn:alpha-Vert}
\alpha_{\Theta^\Ver}(p, V) = \frac{\Omega^\Ver(e_1,\ldots, e_n)^2}{ |\Omega^\Ver(e_1,\ldots, e_n)|^2}
\end{equation}
where $p\in Y'$, $V\cong \bR^n$ is a linear Lagrangian subspace of $(\Ver_p, \omega^\Ver(p))$, and $(e_1,\ldots,e_n)$ is 
an ordered $\bR$-basis of $V$.
\end{definition}

Now we consider the horizontal subbundle. The holomorphic 1-form $dv$ on $Y$ restricts to a nowhere vanishing section of the complex line bundle $\Hor_\bC$ on $Y'$, and 
$$
\Theta^\Hor:= (dv/|dv|)^{\otimes 2}
$$
can be viewed as  a unitary frame of the Hermitian line bundle $\Hor_\bC$. 
\begin{definition}
We define $\alpha_{\Theta^\Hor}: LGr(\Hor, \omega^\Hor)\to U(1)$ by 
\begin{equation}\label{eqn:squared phase Hor}
\alpha_{\Theta^\Hor}(p, V) = \frac{dv(e)^2}{|dv(e)|^2}
\end{equation}
where $p\in Y'$, $V\cong \bR$ is a linear Lagrangian subspace of $(\Hor_p, \omega^\Hor(p))$, and $e$ is any non-zero vector in $V$. 
\end{definition}

On $\bC$, we have that $dz$ is a nowhere vanishing holomorphic 1-form, and $\Theta_\bC = dz^{\otimes 2}$ is a unitary frame of $(T_\bC^*\bC)^{\otimes 2}$ with respect to the standard K\"{a}hler form $\omega_\bC = \frac{i}{2} dz\wedge d\bar{z}$. We use $\Theta_\bC = dz^{\otimes 2}$ to define $\alpha_{\Theta_\bC}: LGr(T\bC, \omega_\bC)\to U(1)$.

In summary, the existence of $\Omega$ allows us to define a quadratic complex volume form, which when evaluated on a Lagrangian subspace at a point, gives the squared phase of the Lagrangian subspace. This allows us to define a notion of a grading on a Lagrangian, which is a lift of the squared phase function, to be defined in the next section.

\subsection{Graded Lagrangians}\label{sec:grade_Lag_paper5}
Given a Lagrangian submanifold $L$ in a symplectic manifold $(M,\omega_M)$, define
$$
s_L: L\to LGr(TM, \omega_M), \quad p\mapsto (p, T_p L). 
$$

\begin{definition}
Given a 4-tuple $(M, \omega_M, J_M, \Theta_M)$, where $(M,\omega_M)$ is a symplectic manifold, $J_M:TM\to TM$ is an $\omega_M$-compatible almost complex structure,
and $\Theta_M$ is a unitary frame of the Hermitian line bundle $\det(T_{\bC}^*Y)^{\otimes 2}$, we have a squared phase map
$\alpha_{\Theta_M}: LGr(TM, \omega_M)\to U(1)$. 
The {\em squared phase function of a Lagrangian $L$} in $M$ is defined to be
$$
\alpha_L:= \alpha_{\Theta_M}\circ s_L: L\to U(1).
$$
A {\em grading} of $L$ is a smooth map $\widetilde{\alpha}_L: L\to \bR$ that lifts $\alpha_L$, i.e.~such that $e^{2\pi i \widetilde{\alpha}_L} = \alpha_L : L\to U(1)$. 
\end{definition}
Suppose that  $L\subset M$ is a connected Lagrangian submanifold. Then there exists a grading of $L$ if and only if $\alpha_L: L\to U(1)$ is homotopic to a constant map, and in this case, given any $p\in L$ and any $\phi \in \bR$ such that $e^{2\pi i \phi}= \alpha_L(p)$, there exists a unique grading $\widetilde{\alpha}_L: L\to \bR$ such that $\widetilde{\alpha}_L(p)=\phi$.

In this paper we consider the following 4-tuples: 
$$
(Y, \omega, J, \Theta),\quad (Y_c, \omega_c, J_c, \Theta_c), \quad (\bC, \omega_\bC = dx\wedge dy, \bj, \Theta_\bC)
$$
where $c\in C$ and $J_c$ is the complex structure on $Y_c$.
\begin{definition}\label{def: alpha L}
    Given a Lagrangian $L$ in $Y$, $Y_c$, or $\bC$, respectively, 
we define $\alpha_L: L\to U(1)$ to be the composition $\alpha_\Theta \circ s_L$, $\alpha_{\Theta_c}\circ s_L$, or $\alpha_{\Theta_\bC}\circ s_L$, respectively. \end{definition}

\begin{definition}
Given a fibered Lagrangian $L\subset Y_C\subset Y$, define 
\begin{eqnarray*}
    & s^\Ver_L: L\to LGr(\Ver,\omega^\Ver), & p\mapsto \left(p, T_p L_{v(p)} \right),  \\
    & s^\Hor_L: L\to LGr(\Hor,\omega^\Hor), & p\mapsto \left(p, (dv_p|_{\Hor_p})^{-1}(T_{v(p)}v(L) \right),
\end{eqnarray*}
where $L_{v(p)} = L\cap Y_{v(p)}$ is a Lagrangian submanifold of $(Y_{v(p)}, \omega_{v(p)})$ and $v(L)$ is a Lagrangian submanifold of $(C, \omega_\bC|_C)$. We observe that
$$
T_p L_{v(p)} = T_p L\cap \Ver_p,\quad 
(dv_p|_{\Hor_p})^{-1}\left(T_{v(p)}v(L)\right)= T_p L\cap \Hor_p.
$$
We define 
$$
\alpha_L^\Ver:= \alpha_{\Theta^\Ver}\circ s_L^\Ver: L\to U(1), \quad
\alpha_L^\Hor:= \alpha_{\Theta^\Hor}\circ s_L^\Hor: L\to U(1).
$$
A {\em relative grading} of $L$ is a smooth map $\widetilde{\alpha}^\Ver_L:L\to \bR$ such that $e^{2\pi i \widetilde{\alpha}_L^\Ver} = \alpha^\Ver_L: L\to U(1)$. 
\end{definition}

These are sensible definitions because they agree with the squared phase functions of the fiber and base Lagrangians, see Equations \eqref{eqn:alpha-Ver} and \eqref{eqn:alpha-Hor} in the next Lemma. Furthermore, similar to how the (squared) determinant of a block diagonal matrix is the product of the (squared) determinants of the blocks, so does the squared phase function of a fibered $L$ split into a product of those of the fiber and base, see Equation \eqref{eqn:Ver-Hor}.

\begin{lemma}\label{lem:grading_split}
Let $L\subset Y_C\subset Y$ be a fibered Lagrangian. Then for any $p\in L$, 
\begin{equation}\label{eqn:alpha-Ver}
\alpha_L^\Ver(p) = \alpha_{L_{v(p)}}(p), 
\end{equation}
\begin{equation} \label{eqn:alpha-Hor}
\alpha_L^\Hor(p) =\alpha_{v(L)}(v(p)), 
\end{equation} 
\begin{equation} \label{eqn:Ver-Hor}
\alpha_L(p) = \alpha_L^\Ver(p) \alpha_L^\Hor(p). 
\end{equation}
\end{lemma}

\begin{proof} 
It is straightforward to check \eqref{eqn:alpha-Ver} and \eqref{eqn:alpha-Hor} from the definitions, by noting that we use the notations $\alpha_{L_{v(p)}}$ and $\alpha_{v(L)}$ in Definition \ref{def: alpha L} with  $L_{v(p)}$ a Lagrangian in $Y_{v(p)}$ and $v(L)$  a Lagrangian in $\bC$. We now check \eqref{eqn:Ver-Hor}.
Let $p\in L$ and let $c=v(p)\in C$.  Let $(e_1,\ldots, e_{n+1})$ be an ordered $\bR$-basis  of $T_p L$ such that $(e_1,\ldots,e_n)$ is an ordered $\bR$-basis
of $T_p L\cap \Ver_p = T_p L_c$ and $e_{n+1} \in  T_p L\cap \Hor_p = (dv_p|_{\Hor_p})^{-1}(T_{v(p)} v(L) )$. Equation \eqref{eqn:Ver-Hor} follows
from the following four equalities:
$$
\alpha_L(p) = \frac{ \Omega(p)(e_1,\ldots, e_{n+1})^2}{ |\Omega(p)(e_1,\ldots, e_{n+1})|^2}, 
$$
$$
\alpha_L^\Ver(p)= \alpha_{L_c}(p) = \frac{ \Omega_c(p)(e_1,\ldots, e_n)^2}{ |\Omega_c(p)(e_1,\ldots, e_n)|^2},  \quad
\alpha_L^\Hor(p)= \alpha_{v(L)}(c) = \frac{ dv_p(e_{n+1})^2 }{ |dv_p(e_{n+1})|^2 },
$$
$$
\Omega(p)(e_1,\ldots, e_{n+1}) =\Omega_c(p)(e_1,\ldots,e_n) dv_p(e_{n+1}).
$$
\end{proof}

Let $L\subset Y_C$ be a connected fibered Lagrangian. 
By \eqref{eqn:alpha-Ver}, if $\widetilde{\alpha}^\Ver_L: L\to \bR$ is a relative grading of $L$ then 
$\widetilde{\alpha}_L^\Ver|_{L_c}: L_c\to \bR$ is a grading of 
$L_c$ for any $c\in C$. Conversely, since the inclusion
$L_c\subset L$ is a deformation retract, given any $c\in C$
and any grading $\widetilde{\alpha}_{L_c}: L_c\to \bR$ of $L_c$, there exists a unique relative grading
$\widetilde{\alpha}_L^\Ver: L\to \bR$ such that $\widetilde{\alpha}^\Ver_L |_{L_c} = \widetilde{\alpha}_{L_c}$. {Thus there is a one-to-one correspondence between relative gradings of $L$ and gradings of $L_c$ for a given $c$, under restriction and inclusion.}

\begin{corollary}\label{cor:rel_to_tot_grade}
Given a relative grading $\widetilde{\alpha}^\Ver_L: L\to \bR$ of a fibered Lagrangian $L\subset Y_C\subset Y$ and a grading $\widetilde{\alpha}_{v(L)}:v(L)\to \bR$ of $v(L)\subset C \subset \bC$, define $\widetilde{\alpha}_L: L\to \bR$ by 
$$
\widetilde{\alpha}_L(p) = \widetilde{\alpha}_L^\Ver(p) +
\widetilde{\alpha}_{v(L)}(v(p)).  
$$
Then $\widetilde{\alpha}_L:L\to \bR$ is a grading of $L$.
\end{corollary}
\begin{proof}
$$
e^{2\pi i\widetilde{\alpha}_L(p)}
= e^{2\pi i \widetilde{\alpha}_L^\Ver(p)}
e^{2\pi i \widetilde{\alpha}_{v(L)}(v(p))}
= \alpha^\Ver_L(p) \alpha_{v(L)}(v(p))
=\alpha_L(p)
$$
where the last equality follows from 
\eqref{eqn:alpha-Hor} and \eqref{eqn:Ver-Hor}.
\end{proof}

\begin{example}[Example of a graded Lagrangian] In order to use Lemma \ref{lem:grading_split} in practice, we need to find a frame $e_1,\ldots,e_n$ for the fiber Lagrangian tangent space. It may be that the coordinates describing the Lagrangian are related by a complicated transform to the coordinates in which $\Omega$ is written (for example action-angle coordinates which trivialize $\omega$ versus holomorphic coordinates which trivialize $J$, related by the Legendre transform). The beauty of this Lemma is that we can still find $\Omega(e_1,\ldots,e_{n+1})$ and hence the squared phase map Equation \eqref{eqn:alpha} in many cases with the help of this fiber and base splitting.

One case is when the tangent bundle of a fiber $T_pY_c=Y_c \times \bR^{2n}$ is trivial, the fibered Lagrangian $L$ has a trivializable tangent bundle, and $\alpha_{L_c}$ is homotopic to a constant map. Then
$$
\Ver|_{Y_c}=Y_c \times \bC^n \implies F(\Ver_\bC|_{Y_c},\omega^\Ver|_{Y_c}) = Y_c \times U(n) \ni (p,[e_1 \mid \ldots \mid e_n]=I_n)
$$
where $\{e_j=(0,\ldots,0,1,0,\ldots,0)\}_{j=1,\ldots,n} \subset \bC^n$ is the frame of the standard basis vectors. (Any other unitary frame corresponds to a unitary matrix.) So elements of
$$
LGr(\Ver|_{Y_c},\omega^\Ver|_{Y_c})=F(\Ver_\bC|_{Y_c},\omega^\Ver|_{Y_c})/O(n)=Y_c \times U(n)/O(n)
$$
are pairs $(p,aO(n))$ of a point $p\in Y_c$ in the fiber and coset $aO(n)$ for $a \in U(n)$. For a symplectic Landau-Ginzburg model, the structure group for the unitary
frame bundle $F(T_{\bC}Y',\omega|_{Y'})$ can be reduced to the subgroup $U(n)\times U(1)\subset U(n+1)$ (see Section \ref{sec:unitary-frame}):
$$
F(T_{\bC}Y', \omega|_{Y'}) = 
F \left( (\omega^\Ver, J^\Ver)\oplus (\omega^\Hor,  J^\Hor) \right) \times_{U(n)\times U(1)} 
U(n+1).
$$
Therefore in this case, $\Theta$ takes the squared determinant of a matrix under this identification of an element in $LGr(\Ver_\bC|_{Y_c},\omega^\Ver|_{Y_c})$ with a representative in $U(n)\times U(1)$. That is
$$
\alpha_{L}: L \xrightarrow[]{s_{L}} LGr(TY,\omega) \xrightarrow[]{\alpha_\Theta} U(1)
$$
is defined in the fiber over $c$ to be
$$
\alpha_L|_{L_c}=\Theta \circ \phi_Y \circ s_{L}|_{L_c}: L_c \xrightarrow[]{s_L} LGr(\Ver_\bC|_{Y_c},\omega^\Ver|_{Y_c})\times_{Y_c} LGr(\Hor_\bC|_{Y_c}, \omega^\Hor|_{Y_c}) \xrightarrow[]{\phi_Y} F(\det(T_\bC Y)^{\otimes 2}) \xrightarrow[]{\Theta}  U(1)
$$
$$
p \xmapsto[]{s_L} (p,T_p L_{v(p)}\oplus (dv_p|_{\Hor_p})^{-1}(T_{v(p)}v(L))) \xmapsto[]{\phi_Y}\left(p,  \det\left( \begin{matrix} a & 0 \\ 0 & dv_p(e)/|dv_p(e)| \end{matrix} \right)^2 \right) \xmapsto[]{\Theta} \det(a)^2 \alpha_{v(L)}(c) 
$$
for unitary matrix $a=a(p)$ whose columns form a frame which trivializes the fiber Lagrangian tangent bundle $T_pL_{v(p)}$ such that $\det(a^2)=\alpha_{L_c}(p)$, and $e=e(p)$ is a nonzero vector in $ \Hor_p\cap T_pL$. Note that we know a lift $\tilde{\alpha}_L$ exists so we can grade the fibered Lagrangian $L$; it is homotopy equivalent to fiber Lagrangian $L_c$ by a deformation retract and $\alpha_{L_c}$ is homotopic to a constant map by assumption, so $\alpha_L:L \to U(1)$ is homotopic to a constant map. Thus there exists an $\widetilde \alpha_L:L \to \bR$ such that $\alpha_{L}=e^{2\pi i \tilde \alpha_{L}}$. 
\end{example}

\subsection{$\bZ$-grading on morphisms}\label{sec:grade_mors_paper5}

Let $\widetilde{L_0}=(L_0,\widetilde{\alpha}_{L_0})$ and $\widetilde{L_1}=(L_1,\widetilde{\alpha}_{L_1})$ be two graded and fibered Lagrangians in $Y_C$. Let $L_{j,c}:=L_j \cap Y_c$ and $p\in L_0 \pitchfork L_1$. There is a linear symplectomorphism 
$$
A=A^\Ver \oplus A^\Hor: T_pY'=\Ver_p \oplus \Hor_p \to \bC^{n+1}
$$
$$
\quad A^\Ver(T_pL_{0,v(p)})=\bR^n, \quad A^\Ver(T_p L_{1,v(p)})=i\bR^n
$$
$$
A^\Hor(dv_p|_{\Hor_p})^{-1}(T_{v(p)}v(L_0) ) = \bR,\quad  A^\Hor(dv_p|_{\Hor_p})^{-1}(T_{v(p)}v(L_1) ) = i\bR.
$$
Then the \emph{canonical short path} is

$$
\la=\la^\Ver\oplus \la^\Hor: [0,1] \to LGr(\Ver_p,\omega^\Ver) \times LGr(\Hor_p,\omega^\Hor) \subset LGr(T_p Y', \omega), \,
$$
$$t \mapsto (A^\Ver\oplus A^\Hor)^{-1}(e^{-\frac{i \pi t}{2}}\bR^{n+1}).
$$
\begin{remark} 
Under the complex linear and symplectic isomorphism $A:T_pY' \to \bC^{n+1}$, the  inclusion 
$$
 LGr(\Ver_p,\omega^\Ver) \times LGr(\Hor_p,\omega^\Hor) \subset LGr(T_p Y', \omega)
 $$
corresponds to
$$
U(n)/O(n) \times U(1)/O(1) = \left(U(n)\times U(1)\right)/\left(O(n)\times O(1)\right) \subset U(n+1)/O(n+1).
$$
\end{remark} 
We define $\alpha_\la:=\alpha_\Theta \circ \la:[0,1]\to U(1)$ and define the lift $\widetilde{\alpha}_\la$ so that  $\widetilde{\alpha}_\la(0)=\widetilde{\alpha}_{L_0}(p)$. In particular, for $\alpha^\Ver_\la := \alpha_{\Theta^\Ver} \circ \la^\Ver$ and $\alpha^\Hor_\la:= \alpha_{\Theta^\Hor} \circ \la^\Hor$ then $\alpha_\la=\alpha^\Ver_\la \alpha^\Hor_\la$ similar to Equation \eqref{eqn:Ver-Hor}. Likewise, define lifts $\widetilde{\alpha}_{\la}^\Ver$ and $\widetilde{\alpha}_{\la}^\Hor$ so that $\widetilde{\alpha}_{\la}^\Ver(0)=\widetilde{\alpha}^\Ver_{L_0}(p)$ and $\widetilde{\alpha}_{\la}^\Hor(0)=\widetilde{\alpha}_{v(L_0)}(v(p))$. The definition of the degree of an intersection point $p\in L_0 \pitchfork L_1$ is then (see \cite[p 7]{ACLLb} for more details) 
\begin{equation}\label{eq:tot_degree}
\deg(\widetilde{L}_0, \widetilde{L}_1;p)=(\widetilde{\alpha}_{L_1}(p)-\cancel{\widetilde{\alpha}_{L_0}(p)})-(\widetilde{\alpha}_\la(1)-\cancel{\widetilde{\alpha}_\la(0)})=[\widetilde{\alpha}_{L_{1}}^\Ver(p) +
\widetilde{\alpha}_{v(L_1)}(v(p))]-[\widetilde{\alpha}_\la^\Ver(1) +
\widetilde{\alpha}_{\la}^\Hor(1)].    
\end{equation}

\begin{theorem}\label{cor_Maslov_paper5}
    Consider two fibered Lagrangians $L_0, L_1 \subset Y_C$ and $p\in L_{0,c} \pitchfork L_{1, c}$, so $c=v(p)$. Assume the fiber Lagrangians $L_{j,c}$ admit gradings for $j=0,1$. Then $L_0, L_1$ admit gradings $\widetilde{L}_0$, $\widetilde{L}_1$ such that
    $$
    \deg(\widetilde{L_0}, \widetilde{L_1};p)=\deg_{Y_c}(\widetilde{L}_{0,c}, \widetilde{L}_{1,c}; p)+\deg_{\bC}(\widetilde{v(L_0)},\widetilde{v(L_1)}; c).
    $$
\end{theorem}

\begin{proof} 
By assumption we have gradings $\widetilde{L}_{j,c}$ for $j=0,1$. Choose a grading $\widetilde{v(L_j)}$ for $j=0,1$ as well. Then grade $L_j$ as in Corollary \ref{cor:rel_to_tot_grade}. By Equation \eqref{eq:tot_degree}, 
\begin{equation}
    \begin{aligned}
        \deg(\widetilde{L}_0, \widetilde{L}_1;p)&=[\widetilde{\alpha}_{L_1}^\Ver(p) +
\widetilde{\alpha}_{v(L_1)}(v(p))]-[\widetilde{\alpha}_\la^\Ver(1) +
\widetilde{\alpha}_{\la}^\Hor(1)]\\
&=[\widetilde{\alpha}_{L_1}^\Ver(p)-\widetilde{\alpha}_\la^\Ver(1)] +
[\widetilde{\alpha}_{v(L_1)}(v(p))-
\widetilde{\alpha}_{\la}^\Hor(1)]\\
&=\deg_{Y_c}(\widetilde{L}_{0,c}, \widetilde{L}_{1,c}; p)+\deg_{\bC}(\widetilde{v(L_0)},\widetilde{v(L_1)}; c).
    \end{aligned}
\end{equation}
\end{proof}
    \begin{lemma}\label{bigon_Maslov_paper5}
    Consider a bigon $u\in \cD(Y, (L_0,L_1), \{p_0=p_-,p_1=p_+\})$ with boundary on two graded and fibered Lagrangians $L_0,L_1 \subset Y_C$. The bigon may contain isolated critical points of $v$ on its interior. Let $c_\pm=\gamma_{L_j}(t_\pm)$ denote the two intersection points in the base as in Figure \ref{fig:deg_pt_base} and $p_\pm$ the intersection points in the total space above them. Then 
    $$
\deg(\widetilde{L_0}, \widetilde{L_1}; p_+) - \deg(\widetilde{L_0}, \widetilde{L_1}; p_-) =\deg(\widetilde{L_{0,c_+}}, \widetilde{ L_{1,c_+}}; p_+)-\deg(\Phi(\widetilde{L_{0,c_+}}), \widetilde{ L_{1,c_+}}; p'_-)+1
$$
where $\Phi$ is the monodromy map $Y_{c_+} \to Y_{c_+}$, which is a parallel transport around a loop that goes once around the critical values in the base enclosed by the bigon, and $p'_-:=(\Phi_{\gamma_{L_1}(t_+) \to \gamma_{L_1}(t_-)})^{-1}(p_-)$.  
\end{lemma}

{We make a remark of interest on the relation to Skein theory in Figure \ref{fig:deg_pt_base_Skein} and Remark \ref{rmk:skein}.}

\begin{remark}\label{rmk:skein} Without the assumption $2c_1(Y)=0$, intersection points still always have a $\bZ/2$-grading that matches the Skein relations; for $c_\pm \in \Hom^*(v(L_0),v(L_1))$, if the input Lagrangian goes above and the output below, then the $-$ crossing in the sense of Skein theory corresponds to the degree 0 intersection point $c_-$ and the $+$ crossing to the degree 1 intersection point $c_+$.  
\end{remark}

\begin{figure}[h]
\centering
\begin{subfigure}[b]{0.45\textwidth}
\centering
\includegraphics[scale=0.25]{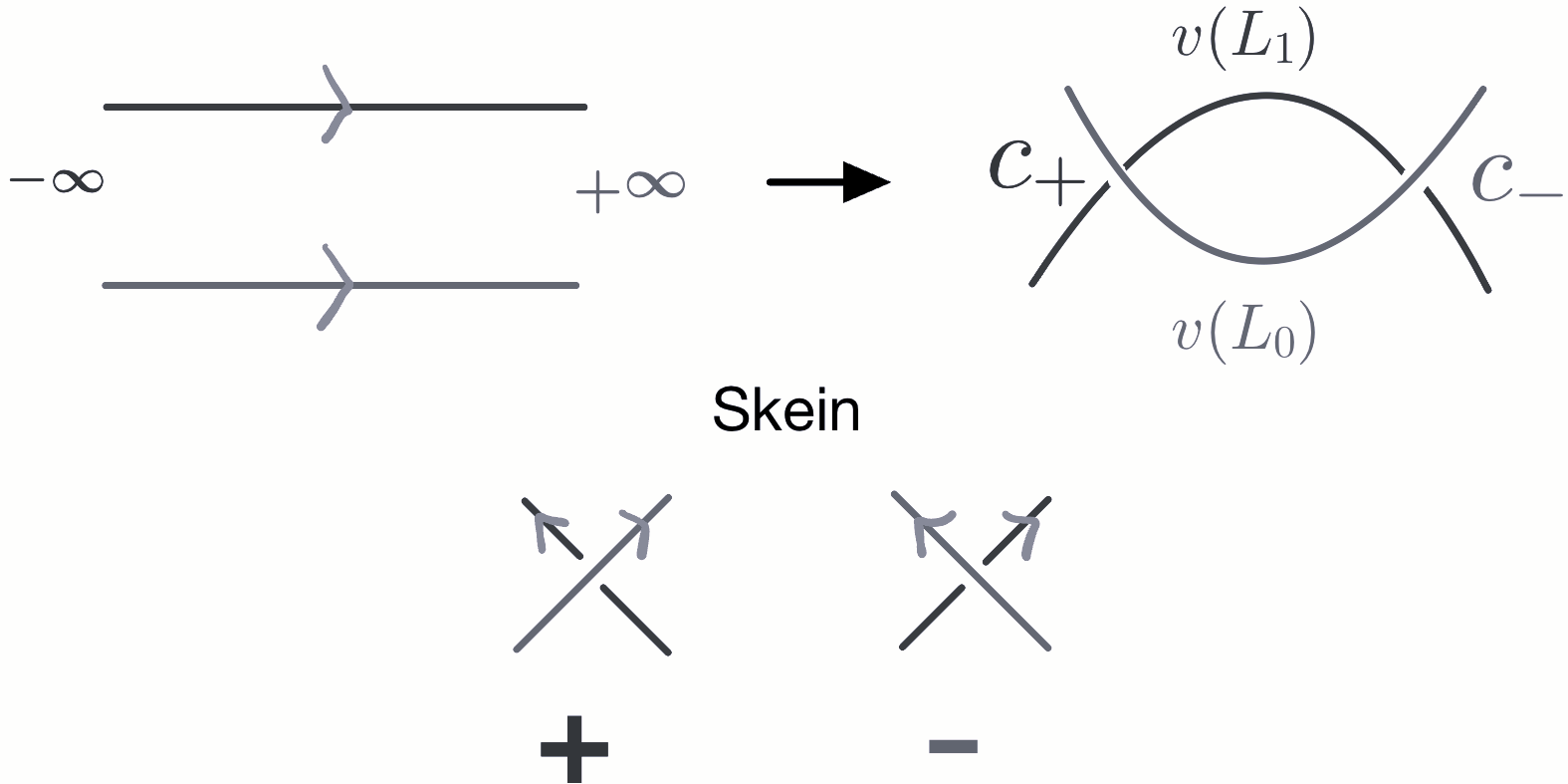}
\subcaption{The equivalence of a choice of $\bZ/2$-grading with the convention in Skein relations.}
\label{fig:deg_pt_base_Skein}
\end{subfigure}
\hfill
\begin{subfigure}[b]{0.4\textwidth}
    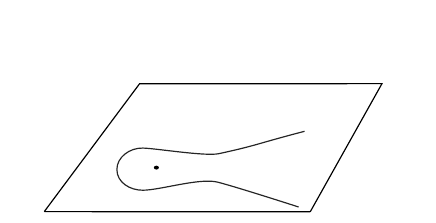
    \subcaption{The differential counts bigons over that in the base. $\Phi$ is the monodromy.}
    \label{fig:deg_pt_base}
\end{subfigure}
\caption{Grading of points in a bigon}
\end{figure}

\begin{proof}[Proof of Lemma \ref{bigon_Maslov_paper5}]
Suppose the $v(L_j)$ are graded by $\phi_j^\Hor(t):= \frac{1}{\pi}\arg(\dot\gamma_{L_j}(t))\in (-1/2,1/2)$. Let $\phi_{j,\pm}^\Hor:=\phi_j^\Hor(t_\pm)$ where $\gamma_{L_j}(t_\pm)=c_\pm$. In other words, these are the angles of the tangent vectors to $v(L_j)$ at the two intersection points $c_\pm$. The canonical short paths in the total space are
\[
\lambda(t) =\begin {cases}
\la_+(t)=\lambda^\Ver_+(t)\oplus (dv_p)^{-1}( e^{-\pi i  (1-(\phi^\Hor_{1,+}-\phi^\Hor_{0,+})) t} \dot \gamma_{L_0}(t_+)) & \text{if } p=p_+ \in Y_{c_+},\\
\\
\la_-(t)= \lambda^\Ver_-(t)\oplus  (dv_p)^{-1}(e^{-\pi i  (\phi^\Hor_{0,-}-\phi^\Hor_{1,-}) t} \dot \gamma_{L_0}(t_-)) & \text{if } p=p_- \in Y_{c_-}
\end{cases}
\]
where $\lambda^\Ver_{\pm}(t)$ are the canonical short paths between fiber Lagrangians over $c_{\pm}$. Note that in $Y_{c_-}$ we have an identification of Floer complexes 
\begin{equation}
CF(\Phi_{\gamma_{L_0}(t_+) \to \gamma_{L_0}(t_-)}(L_{0,c_+}),\Phi_{\gamma_{L_1}(t_+) \to \gamma_{L_1}(t_-)}(L_{1,c_+}))  
\cong CF(\Phi(L_{0,c_+}),L_{1,c_+})
\end{equation}
where $(\Phi_{\gamma_{L_1}(t_+) \to \gamma_{L_1}(t_-)})^{-1}\Phi_{\gamma_{L_0}(t_+) \to \gamma_{L_0}(t_-)}=\Phi$ is the monodromy, by applying the diffeomorphism\newline $(\Phi_{\gamma_{L_1}(t_+) \to \gamma_{L_1}(t_-)})^{-1}$ to both Lagrangians branes on the left Floer complex. Thus by Theorem \ref{cor_Maslov_paper5}, where  $\widetilde{\alpha}_{\la_\pm}^\Hor(0)=\widetilde{\alpha}_{v(L_0)}(c_\pm)=\phi^\Hor_{0,\pm}$ so it cancels out as in Equation \eqref{eq:tot_degree},
\begin{equation}
    \begin{aligned}
\deg(\widetilde{L_0}, \widetilde{L_1}; p_+)& =\deg(\widetilde{L_{0,c_+}}, \widetilde{ L_{1,c_+}}; p_+)+\deg_{\bC}(\widetilde{v(L_0)},\widetilde{v(L_1)}; c_+)\\
& =\deg(\widetilde{L_{0,c_+}}, \widetilde{ L_{1,c_+}}; p_+)+(\phi^\Hor_{1,+}-\phi^\Hor_{0,+})-(-(1-(\phi^\Hor_{1,+}-\phi^\Hor_{0,+})))\\
&=\deg(\widetilde{L_{0,c_+}}, \widetilde{ L_{1,c_+}}; p_+)+1\\
\deg(\widetilde{L_0}, \widetilde{L_1};p'_-) & =\deg({\Phi(\widetilde{L_{0,c_+}})}, \widetilde{ L_{1,c_+}}; p'_-)+(\phi^\Hor_{1,-}-\phi^\Hor_{0,-})-(-(\phi^\Hor_{0,-}-\phi^\Hor_{1,-}))\\
&= \deg({\Phi(\widetilde{L_{0,c_+}})}, \widetilde{ L_{1,c_+}}; p'_-).        
    \end{aligned}
\end{equation}
 Although different graded lifts exist $\widetilde{\alpha_{v(L)}}:v(L) \to \bR$ for $L=L_0$ and $L=L_1$, the difference between degrees of $p_+$ and $p_-$ stays the same.
\end{proof}

\begin{remark} The bigon may pass through  singular fibers, but we only consider the Maslov index of a loop of matrices at an intersection point away from the singular fiber, where we use a decomposition into base and fiber with the globally defined holomorphic $(n+1)$-form that splits into base and fiber terms, above.
\end{remark}

\begin{remark}
Note that the definition of a grading of an intersection point is consistent with the Morse theory interpretation of $J$-holomorphic strips. Lagrangian Floer cohomology is Morse cohomology for an action functional measuring negative area \cite[p 3]{fuk_intro}. Critical points of the functional are intersection points $L_0 \cap L_1$. Counting flow lines agrees with the Floer differential counting strips that flow with increasing signed $(\dd_s, \dd_t)$ area for coordinates $(s,t)\in \bR \times [0,1]$ on the strip in Figure \ref{fig:deg_pt_base_Skein}. Therefore, since we are considering \underline{co}homology instead of Morse homology, the differential goes opposite to the flow direction, and the intersection points are graded accordingly to increase degree by 1 from $p_-$ to $p_+$. This matches the result in Lemma \ref{bigon_Maslov_paper5}. (The flow direction, grading of a Lagrangian, and orientation of a Lagrangian are not directly related.) 
\end{remark}

\bibliographystyle{amsalpha}
\bibliography{glob_gen2_hms}

\end{document}

%% file: ushapeBase_paper5_grayscale.pdf_tex
\begingroup%
  \makeatletter%
  \providecommand\color[2][]{%
    \errmessage{(Inkscape) Color is used for the text in Inkscape, but the package 'color.sty' is not loaded}%
    \renewcommand\color[2][]{}%
  }%
  \providecommand\transparent[1]{%
    \errmessage{(Inkscape) Transparency is used (non-zero) for the text in Inkscape, but the package 'transparent.sty' is not loaded}%
    \renewcommand\transparent[1]{}%
  }%
  \providecommand\rotatebox[2]{#2}%
  \newcommand*\fsize{\dimexpr\f@size pt\relax}%
  \newcommand*\lineheight[1]{\fontsize{\fsize}{#1\fsize}\selectfont}%
  \ifx\svgwidth\undefined%
    \setlength{\unitlength}{361.84617891bp}%
    \ifx\svgscale\undefined%
      \relax%
    \else%
      \setlength{\unitlength}{\unitlength * \real{\svgscale}}%
    \fi%
  \else%
    \setlength{\unitlength}{\svgwidth}%
  \fi%
  \global\let\svgwidth\undefined%
  \global\let\svgscale\undefined%
  \makeatother%
  \begin{picture}(1,0.48765425)%
    \lineheight{1}%
    \setlength\tabcolsep{0pt}%
    \put(0,0){\includegraphics[width=\unitlength,page=1]{ushapeBase_paper5_grayscale.pdf}}%
    \put(0.56786505,0.1800561){\makebox(0,0)[lt]{\lineheight{1.25}\smash{\begin{tabular}[t]{l}\textcolor{black}{$v(L_1)$}\end{tabular}}}}%
    \put(0.48858906,0.2811285){\makebox(0,0)[lt]{\lineheight{1.25}\smash{\begin{tabular}[t]{l}\textcolor{black}{$v(L_0)$}\end{tabular}}}}%
    \put(0,0){\includegraphics[width=\unitlength,page=2]{ushapeBase_paper5_grayscale.pdf}}%
    \put(0.17856632,0.14736676){\color[rgb]{0,0,0}\makebox(0,0)[lt]{\lineheight{1.25}\smash{\begin{tabular}[t]{l}$c_+$\end{tabular}}}}%
    \put(0.42492505,0.1531954){\color[rgb]{0,0,0}\makebox(0,0)[lt]{\lineheight{1.25}\smash{\begin{tabular}[t]{l}$c_-$\end{tabular}}}}%
  \end{picture}%
\endgroup%

%% file: pyramid_paper5_grayscale.pdf_tex
\begingroup%
  \makeatletter%
  \providecommand\color[2][]{%
    \errmessage{(Inkscape) Color is used for the text in Inkscape, but the package 'color.sty' is not loaded}%
    \renewcommand\color[2][]{}%
  }%
  \providecommand\transparent[1]{%
    \errmessage{(Inkscape) Transparency is used (non-zero) for the text in Inkscape, but the package 'transparent.sty' is not loaded}%
    \renewcommand\transparent[1]{}%
  }%
  \providecommand\rotatebox[2]{#2}%
  \newcommand*\fsize{\dimexpr\f@size pt\relax}%
  \newcommand*\lineheight[1]{\fontsize{\fsize}{#1\fsize}\selectfont}%
  \ifx\svgwidth\undefined%
    \setlength{\unitlength}{578.23837438bp}%
    \ifx\svgscale\undefined%
      \relax%
    \else%
      \setlength{\unitlength}{\unitlength * \real{\svgscale}}%
    \fi%
  \else%
    \setlength{\unitlength}{\svgwidth}%
  \fi%
  \global\let\svgwidth\undefined%
  \global\let\svgscale\undefined%
  \makeatother%
  \begin{picture}(1,0.44822856)%
    \lineheight{1}%
    \setlength\tabcolsep{0pt}%
    \put(0,0){\includegraphics[width=\unitlength,page=1]{pyramid_paper5_grayscale.pdf}}%
    \put(0.24384556,0.16739521){\makebox(0,0)[lt]{\lineheight{1.25}\smash{\begin{tabular}[t]{l}\textcolor{black}{$v(L_0)$}\end{tabular}}}}%
    \put(0.27769,0.12052037){\makebox(0,0)[lt]{\lineheight{1.25}\smash{\begin{tabular}[t]{l}\textcolor{black}{$v(L_1)$}\end{tabular}}}}%
    \put(0.20979743,0.09419685){\makebox(0,0)[lt]{\lineheight{1.25}\smash{\begin{tabular}[t]{l}\textcolor{black}{$\delta=v( u(\dd_0 \bD))$}\end{tabular}}}}%
    \put(0.23092417,0.03178959){\makebox(0,0)[lt]{\lineheight{1.25}\smash{\begin{tabular}[t]{l}\textcolor{black}{$v(L_2)$}\end{tabular}}}}%
    \put(0,0){\includegraphics[width=\unitlength,page=2]{pyramid_paper5_grayscale.pdf}}%
    \put(0.43872589,0.1760169){\makebox(0,0)[lt]{\lineheight{1.25}\smash{\begin{tabular}[t]{l}\textcolor{black}{$h_s(\bR)=v(\psi_s(L_0))$}\end{tabular}}}}%
    \put(0,0){\includegraphics[width=\unitlength,page=3]{pyramid_paper5_grayscale.pdf}}%
    \put(0.71372169,0.17279904){\makebox(0,0)[lt]{\lineheight{1.25}\smash{\begin{tabular}[t]{l}\textcolor{black}{$h_1(\bR)=v(\psi_1(L_0))$}\end{tabular}}}}%
    \put(0,0){\includegraphics[width=\unitlength,page=4]{pyramid_paper5_grayscale.pdf}}%
    \put(0.48809011,0.31545252){\makebox(0,0)[lt]{\lineheight{1.25}\smash{\begin{tabular}[t]{l}$[u]$\end{tabular}}}}%
    \put(0.48294116,0.43107387){\makebox(0,0)[lt]{\lineheight{1.25}\smash{\begin{tabular}[t]{l}$[u']-[u]$\end{tabular}}}}%
    \put(0,0){\includegraphics[width=\unitlength,page=5]{pyramid_paper5_grayscale.pdf}}%
    \put(0.37066493,0.34793845){\makebox(0,0)[lt]{\lineheight{1.25}\smash{\begin{tabular}[t]{l}$[u'']$\end{tabular}}}}%
    \put(0,0){\includegraphics[width=\unitlength,page=6]{pyramid_paper5_grayscale.pdf}}%
    \put(0.09441048,0.09137446){\makebox(0,0)[lt]{\lineheight{1.25}\smash{\begin{tabular}[t]{l}$-\epsilon$\end{tabular}}}}%
    \put(0.6256118,0.09452533){\makebox(0,0)[lt]{\lineheight{1.25}\smash{\begin{tabular}[t]{l}$-\epsilon$\end{tabular}}}}%
    \put(0,0){\includegraphics[width=\unitlength,page=7]{pyramid_paper5_grayscale.pdf}}%
  \end{picture}%
\endgroup%

%% file: L0_move_paper5_grayscale.pdf_tex
\begingroup%
  \makeatletter%
  \providecommand\color[2][]{%
    \errmessage{(Inkscape) Color is used for the text in Inkscape, but the package 'color.sty' is not loaded}%
    \renewcommand\color[2][]{}%
  }%
  \providecommand\transparent[1]{%
    \errmessage{(Inkscape) Transparency is used (non-zero) for the text in Inkscape, but the package 'transparent.sty' is not loaded}%
    \renewcommand\transparent[1]{}%
  }%
  \providecommand\rotatebox[2]{#2}%
  \newcommand*\fsize{\dimexpr\f@size pt\relax}%
  \newcommand*\lineheight[1]{\fontsize{\fsize}{#1\fsize}\selectfont}%
  \ifx\svgwidth\undefined%
    \setlength{\unitlength}{525.90394738bp}%
    \ifx\svgscale\undefined%
      \relax%
    \else%
      \setlength{\unitlength}{\unitlength * \real{\svgscale}}%
    \fi%
  \else%
    \setlength{\unitlength}{\svgwidth}%
  \fi%
  \global\let\svgwidth\undefined%
  \global\let\svgscale\undefined%
  \makeatother%
  \begin{picture}(1,0.18697683)%
    \lineheight{1}%
    \setlength\tabcolsep{0pt}%
    \put(0,0){\includegraphics[width=\unitlength,page=1]{L0_move_paper5_grayscale.pdf}}%
    \put(0.11956535,0.00249171){\makebox(0,0)[lt]{\lineheight{1.25}\smash{\begin{tabular}[t]{l}$z_0$\end{tabular}}}}%
    \put(0.17581795,0.04613376){\makebox(0,0)[lt]{\lineheight{1.25}\smash{\begin{tabular}[t]{l}$z_0'$\end{tabular}}}}%
    \put(0.17319476,0.13807417){\makebox(0,0)[lt]{\lineheight{1.25}\smash{\begin{tabular}[t]{l}$z'_1$\end{tabular}}}}%
    \put(0.13153826,0.17609889){\makebox(0,0)[lt]{\lineheight{1.25}\smash{\begin{tabular}[t]{l}$z_1$\end{tabular}}}}%
    \put(0,0){\includegraphics[width=\unitlength,page=2]{L0_move_paper5_grayscale.pdf}}%
    \put(0.18525408,0.09439861){\makebox(0,0)[lt]{\lineheight{1.25}\smash{\begin{tabular}[t]{l}\textcolor{black}{$\dd_0'\bD$}\end{tabular}}}}%
    \put(0.23922297,0.11640352){\makebox(0,0)[lt]{\lineheight{1.25}\smash{\begin{tabular}[t]{l}$\phi_1 \cup_{\overline{z_0z_1}} \phi_2$\end{tabular}}}}%
    \put(0.03915587,0.17152685){\makebox(0,0)[lt]{\lineheight{1.25}\smash{\begin{tabular}[t]{l}\textcolor{black}{$\dd_1'\bD$}\end{tabular}}}}%
    \put(-0.00119822,0.09162334){\makebox(0,0)[lt]{\lineheight{1.25}\smash{\begin{tabular}[t]{l}$z_2$\end{tabular}}}}%
    \put(0.03322279,0.01110533){\makebox(0,0)[lt]{\lineheight{1.25}\smash{\begin{tabular}[t]{l}\textcolor{black}{$\dd_2'\bD$}\end{tabular}}}}%
    \put(0,0){\includegraphics[width=\unitlength,page=3]{L0_move_paper5_grayscale.pdf}}%
    \put(0.47452442,0.1740985){\makebox(0,0)[lt]{\lineheight{1.25}\smash{\begin{tabular}[t]{l}$z_1$\end{tabular}}}}%
    \put(0.47404309,0.01042317){\makebox(0,0)[lt]{\lineheight{1.25}\smash{\begin{tabular}[t]{l}$z_0$\end{tabular}}}}%
    \put(0,0){\includegraphics[width=\unitlength,page=4]{L0_move_paper5_grayscale.pdf}}%
    \put(0.68630179,0.04956582){\makebox(0,0)[lt]{\lineheight{1.25}\smash{\begin{tabular}[t]{l}\textcolor{black}{$\gamma_{L_2}$}\end{tabular}}}}%
    \put(0.6964181,0.14335188){\makebox(0,0)[lt]{\lineheight{1.25}\smash{\begin{tabular}[t]{l}\textcolor{black}{$\gamma_{L_1}$}\end{tabular}}}}%
    \put(0,0){\includegraphics[width=\unitlength,page=5]{L0_move_paper5_grayscale.pdf}}%
    \put(0.65923843,0.09557391){\makebox(0,0)[lt]{\lineheight{1.25}\smash{\begin{tabular}[t]{l}$c$\end{tabular}}}}%
    \put(0,0){\includegraphics[width=\unitlength,page=6]{L0_move_paper5_grayscale.pdf}}%
    \put(0.33584422,0.09389588){\makebox(0,0)[lt]{\lineheight{1.25}\smash{\begin{tabular}[t]{l}$z_2$\end{tabular}}}}%
    \put(0.56158028,0.11194496){\makebox(0,0)[lt]{\lineheight{1.25}\smash{\begin{tabular}[t]{l}$v\circ u'$\end{tabular}}}}%
    \put(0,0){\includegraphics[width=\unitlength,page=7]{L0_move_paper5_grayscale.pdf}}%
  \end{picture}%
\endgroup%

%% file: ushape_paper5_grayscale.pdf_tex
\begingroup%
  \makeatletter%
  \providecommand\color[2][]{%
    \errmessage{(Inkscape) Color is used for the text in Inkscape, but the package 'color.sty' is not loaded}%
    \renewcommand\color[2][]{}%
  }%
  \providecommand\transparent[1]{%
    \errmessage{(Inkscape) Transparency is used (non-zero) for the text in Inkscape, but the package 'transparent.sty' is not loaded}%
    \renewcommand\transparent[1]{}%
  }%
  \providecommand\rotatebox[2]{#2}%
  \newcommand*\fsize{\dimexpr\f@size pt\relax}%
  \newcommand*\lineheight[1]{\fontsize{\fsize}{#1\fsize}\selectfont}%
  \ifx\svgwidth\undefined%
    \setlength{\unitlength}{208.95825243bp}%
    \ifx\svgscale\undefined%
      \relax%
    \else%
      \setlength{\unitlength}{\unitlength * \real{\svgscale}}%
    \fi%
  \else%
    \setlength{\unitlength}{\svgwidth}%
  \fi%
  \global\let\svgwidth\undefined%
  \global\let\svgscale\undefined%
  \makeatother%
  \begin{picture}(1,0.51123883)%
    \lineheight{1}%
    \setlength\tabcolsep{0pt}%
    \put(0,0){\includegraphics[width=\unitlength,page=1]{ushape_paper5_grayscale.pdf}}%
    \put(0.09704027,0.42021997){\color[rgb]{0,0,0}\makebox(0,0)[lt]{\lineheight{1.25}\smash{\begin{tabular}[t]{l}\small \\$CF(L_{0,c_+},L_{1,c_+})$\end{tabular}}}}%
    \put(0.02848654,0.35289515){\color[rgb]{0,0,0}\makebox(0,0)[lt]{\lineheight{1.25}\smash{\begin{tabular}[t]{l}$Y$\end{tabular}}}}%
    \put(0,0){\includegraphics[width=\unitlength,page=2]{ushape_paper5_grayscale.pdf}}%
    \put(0.03003806,0.1024462){\color[rgb]{0,0,0}\makebox(0,0)[lt]{\lineheight{1.25}\smash{\begin{tabular}[t]{l}$\mathbb C$\end{tabular}}}}%
    \put(-0.00195484,0.23251373){\color[rgb]{0,0,0}\makebox(0,0)[lt]{\lineheight{1.25}\smash{\begin{tabular}[t]{l}$v$\end{tabular}}}}%
    \put(0,0){\includegraphics[width=\unitlength,page=3]{ushape_paper5_grayscale.pdf}}%
    \put(0.37522593,0.496408){\makebox(0,0)[lt]{\lineheight{1.25}\smash{\begin{tabular}[t]{l}\small \\$\cong CF(\Phi(L_{0,c_+}), L_{1,c_+})[-1]$\end{tabular}}}}%
    \put(0,0){\includegraphics[width=\unitlength,page=4]{ushape_paper5_grayscale.pdf}}%
    \put(0.5708952,0.14395713){\makebox(0,0)[lt]{\lineheight{1.25}\smash{\begin{tabular}[t]{l}\textcolor{black}{$v(L_1)$}\end{tabular}}}}%
    \put(0.49410015,0.26364688){\makebox(0,0)[lt]{\lineheight{1.25}\smash{\begin{tabular}[t]{l}\textcolor{black}{$v(L_0)$}\end{tabular}}}}%
    \put(0.21079711,0.09607335){\color[rgb]{0,0,0}\makebox(0,0)[lt]{\lineheight{1.25}\smash{\begin{tabular}[t]{l}$c_+$\end{tabular}}}}%
    \put(0.45093037,0.11558995){\color[rgb]{0,0,0}\makebox(0,0)[lt]{\lineheight{1.25}\smash{\begin{tabular}[t]{l}$c_-$\end{tabular}}}}%
  \end{picture}%
\endgroup%